\documentclass[12pt]{amsart}
\usepackage[a4paper, total={156mm, 224mm}, centering]{geometry}
\numberwithin{equation}{section}

\usepackage[T1]{fontenc}
\usepackage{lmodern}
\usepackage{amsmath,amssymb,amsthm,mathtools}
\usepackage{enumitem}
\usepackage[hidelinks]{hyperref}
\usepackage[expansion=false]{microtype}

\newtheorem{theorem}{Theorem}[section]
\newtheorem{proposition}[theorem]{Proposition}
\newtheorem{lemma}[theorem]{Lemma}

\newtheorem{conjecture}[theorem]{Conjecture}

\theoremstyle{definition}
\newtheorem{definition}[theorem]{Definition}
\newtheorem{example}[theorem]{Example}
\theoremstyle{remark}

\def\A{\mathbb{A}}
\def\P{\mathbb{P}}
\def\C{\mathbb{C}}
\def\Q{\mathbb{Q}}
\def\R{\mathbb{R}}
\def\Z{\mathbb{Z}}
\def\F{\mathbb{F}}
\def\sO{\mathcal{O}}
\def\sM{\mathcal{M}}
\def\sP{\mathcal{P}}
\def\sX{\mathcal{X}}
\def\sA{\mathcal{A}}
\def\sR{\mathcal{R}}
\def\sW{\mathcal{W}}
\def\sT{\mathcal{T}}
\def\sL{\mathcal{L}}
\def\sH{\mathcal{H}}
\def\sN{\mathcal{N}}
\def\sD{\mathcal{D}}
\def\sI{\mathcal{I}}
\def\ba{\mathbf{a}}
\def\bb{\mathbf{b}}
\def\bc{\mathbf{c}}
\def\bz{\mathbf{z}}
\def\Poly{\operatorname{Poly}}
\def\MPoly{\operatorname{MPoly}}
\def\Rat{\operatorname{Rat}}
\def\PGL{\operatorname{PGL}}
\def\Crit{\operatorname{Crit}}
\def\Gal{\operatorname{Gal}}
\def\Spec{\operatorname{Spec}}
\def\Sym{\operatorname{Sym}}
\def\red{\operatorname{red}}
\def\Supp{\operatorname{Supp}}
\def\id{\operatorname{id}}
\def\pr{\operatorname{pr}}
\def\FL{\operatorname{FL}}
\def\Irr{\operatorname{Irr}}
\def\olQ{\overline{\Q}}
\def\Kab{K^{\mathrm{ab}}}
\def\Lab{L^{\mathrm{ab}}}
\def\Kc{K^{\mathrm{c}}}
\def\Lc{L^{\mathrm{c}}}
\def\dd{\,\mathrm{d}}

\newcommand{\Rom}[1]{\uppercase\expandafter{\romannumeral #1\relax}}

\title[Bogomolov property for polynomial moduli over abelian extensions]{A Bogomolov property for moduli spaces of polynomials over abelian extensions}
\date{October 11, 2026}

\author{Geng-Rui Zhang}
\address{School of Mathematical Sciences, Peking University, Beijing 100871, China}
\email{grzhang@stu.pku.edu.cn}
\email{chibasei@163.com}

\begin{document}

\begin{abstract}
Let $d\geq2$ be an integer, and let $\MPoly^d$ be the moduli space of degree-$d$ polynomials. For every number field $K$, we prove that there exists $\epsilon_{K,d}>0$ such that
\[
\left\lbrace\alpha\in\MPoly^d(\Kab)\colon h_{\mathrm{crit}}(\alpha)<\epsilon_{K,d}\right\rbrace=\left\lbrace\alpha\in\MPoly^d(\Kab)\colon h_{\mathrm{crit}}(\alpha)=0\right\rbrace
\]
is finite, where $h_{\mathrm{crit}}$ is the critical height. In particular, only finitely many $\Kab$-rational points of $\MPoly^d$ are postcritically finite. The proof combines the correspondence method of Ji--Song--Xie with local ramification estimates and the uniform finiteness of intertwined classes in $\MPoly^d$. We also prove an analogous finiteness theorem for non-preperiodic marked points in non-isotrivial one-parameter polynomial families.
\end{abstract}


\maketitle
\tableofcontents

\section{Introduction}
\subsection{Main theorem}
Fix an algebraic closure $\olQ\subseteq\C$. For a number field $K$, let $\Kab$ denote its maximal abelian extension in $\olQ$. A rational map $f\in\C(z)$ of degree $\geq2$ is called \emph{postcritically finite} (PCF) if every critical point of $f$ in $\P^1(\C)$ is preperiodic. If $f\in\C[z]$ is a polynomial, then infinity is a fixed critical point, and $f$ is PCF if all its critical points in $\C$ are preperiodic. Let $\MPoly^d$ be the moduli space of degree-$d$ polynomials under affine conjugacy, which is an irreducible affine variety of dimension $d-1$, defined over $\Q$; see \cite[\S~2.1]{FG22}.

Let $f\in\olQ(z)$ be a rational map of degree $d\geq2$. The \emph{canonical height} of $f$ is
\[
\hat{h}_f:\P^1(\olQ)\to\R_{\geq0},\,x\mapsto\lim_{n\to\infty}\frac{1}{d^n}h(f^{\circ n}(x)),
\]
where $h$ is the standard absolute logarithmic height on $\P^1(\olQ)$. For every $x\in\P^1(\olQ)$, $\hat{h}_f(x)=0$ if and only if $x$ is $f$-preperiodic; see \cite[Theorem~3.22]{Silverman2007}.

For $x\in\P^1(\olQ)$, let $\deg_f(x)$ be the local degree of $f$ at $x$. Define the \emph{critical height} of $f$ by
\[
h_{\mathrm{crit}}(f)=\sum_{c\in\Crit(f)}(\deg_f(c)-1)\hat{h}_f(c)\geq0,
\]
where $\Crit(f)$ is the set of critical points of $f$ in $\P^1(\olQ)$. Thus, $h_{\mathrm{crit}}(f)=0$ if and only if $f$ is PCF. If $f\in\olQ[z]$ is a polynomial, then $\hat{h}_f(\infty)=0$ and
\[
h_{\mathrm{crit}}(f)=\sum_{c\in\Crit_{\mathrm{fin}}(f)}(\deg_f(c)-1)\hat{h}_f(c),
\]
where $\Crit_{\mathrm{fin}}(f)$ is the set of finite critical points of $f$. Since the critical height is invariant under affine conjugacy, it defines a function on $\MPoly^d$:
\[
h_{\mathrm{crit}}:\MPoly^d(\olQ)\to\R_{\geq0},\,[f]\mapsto h_{\mathrm{crit}}(f).
\]

Our main theorem establishes a Bogomolov-type property for $h_{\mathrm{crit}}$ on $\MPoly^d(\Kab)$.
\begin{theorem}\label{thm:main}
Let $d\geq2$ be an integer and $K$ be a number field. There exists a constant $\epsilon_{K,d}>0$ such that the set
\[
\left\lbrace\alpha\in\MPoly^d(\Kab)\colon h_{\mathrm{crit}}(\alpha)<\epsilon_{K,d}\right\rbrace=\left\lbrace\alpha\in\MPoly^d(\Kab)\colon h_{\mathrm{crit}}(\alpha)=0\right\rbrace
\]
is finite.
\end{theorem}
More precisely, Theorem~\ref{thm:main} shows the strong Bogomolov property for $h_{\mathrm{crit}}$ on $\MPoly^d$ over $\Kab$: the zero-height locus is finite, and the positive values are bounded away from zero (cf.~\cite{FM15}).

Theorem~\ref{thm:main} immediately implies the following consequence.
\begin{theorem}\label{thm:pcf}
Let $d\geq2$ be an integer and $K$ be a number field. Then
\[
\sP_d(\Kab):=\left\lbrace[f]\in\MPoly^d(\Kab)\colon f\text{ is PCF}\right\rbrace
\]
is finite.
\end{theorem}
A key obstruction is that the extension $\Kab/\Q$ is infinite. Ingram proves that $h_{\mathrm{crit}}$ is comparable to an ample Weil height on $\MPoly^d$ \cite[Theorem~1.1]{Ingram2012}. Northcott's theorem then yields finiteness for PCF polynomial classes of bounded degree $[\Q([f]):\Q]\leq B$; however, this argument does not apply to the infinite extension $\Kab$.

\subsection{Polynomial families with one marked point}
The methods developed in this paper also apply to one-parameter polynomial families with a marked point.
\begin{theorem}\label{thm:pair}
Let $d\geq2$ be an integer and $B$ be a geometrically integral affine curve defined over a number field $K$. Let
\[
f:B\times\A^1\to B\times\A^1,\,(t,x)\mapsto(t,f_t(x))
\]
be a family of degree-$d$ polynomials parametrized by $B$, defined over $K$, such that the induced morphism $\varphi:B_{\olQ}\to\MPoly^d_{\olQ},\,t\mapsto[f_t]$ is nonconstant. Let $a$ be a marked point of $f$ defined over $K$ which is not $f$-preperiodic; that is, $a\in K[B]=\Gamma(B,\sO_B)$ satisfies
\[
f^{\circ m}(a)\neq f^{\circ n}(a)\quad\text{in }K[B]
\]
for all integers $m>n\geq0$. Then there exists a constant $\epsilon=\epsilon(K,B,f,a)>0$ such that
\[
\left\lbrace t\in B(\Kab)\colon\hat{h}_{f_t}(a(t))<\epsilon\right\rbrace=\left\lbrace t\in B(\Kab)\colon\hat{h}_{f_t}(a(t))=0\right\rbrace
\]
is finite. In particular,
\[
\left\lbrace t\in B(\Kab)\colon a(t)\text{ is }f_t\text{-preperiodic}\right\rbrace
\]
is finite.
\end{theorem}
We present a simple example illustrating Theorem~\ref{thm:pair}.
\begin{example}
Let $d=2$ and $B=\A^1$. Define $f_t(z)=z^2+t$ and $a(t)=t+1$ for $t\in B$, giving a family $f$ of degree-$2$ polynomials parametrized by $B$ with a marked point $a$. The induced morphism $B\to\MPoly^2$ is an isomorphism. Since $\deg_t f_t^{\circ n}(a(t))=2^n$, the marked point $a$ is not $f$-preperiodic. By Theorem~\ref{thm:pair}, for every number field $K$, the set
\[
\left\lbrace t\in\Kab\colon t+1\text{ is }f_t\text{-preperiodic}\right\rbrace
\]
is finite. For $t=-3\in\Q^{\mathrm{ab}}$, the polynomial $f_{-3}(z)=z^2-3$ is not PCF and $a(-3)=-2$ is $f_{-3}$-periodic. So the finiteness does not follow directly from Theorem~\ref{thm:main}.
\end{example}

\subsection{Notation and conventions}
(1) For $n\geq0$ and a map $f:X\to X$, we write $f^{\circ n}$ for the $n$-th iterate, with $f^{\circ0}=\id_X$.

\smallskip

(2) For a number field $L$, $M_L$ denotes the set of places of $L$. For a subset $S\subseteq M_L$, we write $\sO_{L,S}$ for the ring of $S$-integers in $L$. For a place $v\in M_L$, let $\C_v$ be the completion of an algebraic closure of $L_v$. Absolute values extend the usual complex absolute value or the $p$-adic absolute value normalized by $|p|_p=p^{-1}$. For an algebraic extension $E/\Q$, let $\sO_E$ denote the integral closure of $\Z$ in $E$.

\smallskip

(3) We follow \cite{YZ26} for adelic line bundles and their positivity, with all heights normalized over $\Q$.

\smallskip

(4) A variety is a reduced separated scheme of finite type over a field. All subvarieties are closed and reduced. Numerical equivalence of line bundles, denoted $L_1\equiv L_2$, means that $\deg(L_1|_C)=\deg(L_2|_C)$ for every integral projective curve $C$ after base change to an algebraically closed field.

\smallskip

(5) A \emph{bi-finite correspondence} on an integral variety $X$ is a reduced closed subset of $X\times X$ whose irreducible components are finite and surjective over $X$ under both projections.

\subsection{Strategy of the proof}\label{sec:skt}
We give a detailed sketch of the proof of Theorem~\ref{thm:main}. The main new ingredient is combining the intertwined relation (\S~\ref{sec:int}) with separate ramification estimates for the wild and tame cases (Lemma~\ref{lem:in}).

Suppose that the conclusion of Theorem~\ref{thm:main} fails for some number field $K$ and some $d\geq2$. Then there is a sequence $(a_j)_{j=1}^\infty\subseteq\MPoly^d(\Kab)$ of pairwise distinct points with
\[
\lim_{j\to\infty}h_{\mathrm{crit}}(a_j)=0.
\]
After replacing $(a_j)_{j=1}^\infty$ with a subsequence and passing to a finite extension $L/K$, there is a positive-dimensional geometrically integral subvariety $V\subseteq\MPoly^d_L$ such that the sequence $(a_j)_{j=1}^\infty$ is \emph{generic} in $V/\olQ$; that is, every proper closed subset of $V_{\olQ}$ contains only finitely many terms of $(a_j)_{j=1}^\infty$. Since $\Kab\subseteq\Lab$, all $a_j$ belong to $V(\Lab)$.

The first step is to encode the critical-orbit data in a universal family. In Proposition~\ref{prop:sys}, we construct an affine variety $X_d$, a morphism $p_d:X_d\to\MPoly^d$ with a section $s_d$, and a finite endomorphism $F_d:X_d\to X_d$ with $p_d\circ F_d=p_d$ such that
\[
\alpha\text{ is PCF}\ \Longleftrightarrow\ s_d(\alpha)\text{ is }F_d\text{-preperiodic}\quad(\alpha\in\MPoly^d(\C)).
\]
Applying the Noetherian orbit-closure lemma (Lemma~\ref{lem:orb}) to the restriction over $V$, there are two cases. In the first case, there are integers $b>a\geq0$ such that
\[
F_d^{\circ a}\circ s_d=F_d^{\circ b}\circ s_d
\]
on $V$; then every point of $V(\C)$ is PCF, contradicting Thurston's rigidity. Thus, we are in the second case: there are an irreducible component $Y_0$ of the stabilized orbit closure and an integer $l\geq1$ such that
\[
p_d(Y_0)=V,\quad\dim Y_0>\dim V,\quad\text{and}\quad F_d^{\circ l}(Y_0)=Y_0.
\]
Let $\nu:Y\to Y_0$ be the normalization. The endomorphism $F_d^{\circ l}|_{Y_0}$ lifts to a finite endomorphism $G:Y\to Y$. Setting $\varphi=p_d|_{Y_0}\circ\nu:Y\to V$, one has $\varphi\circ G=\varphi$ and $s:=\dim Y-\dim V>0$. The relative-degree formula of Proposition~\ref{prop:deg} shows $\deg G=d^{sl}$.

In Theorem~\ref{thm:met}, we construct a nef adelic line bundle $\overline{A}$ on a normal projective compactification $\overline{X}_d\supseteq X_d$, with $A$ ample. For all $\alpha\in\MPoly^d(\olQ)$ and $n,m\geq0$, one has
\[
k_d d^{n+m}h_{\mathrm{crit}}(\alpha)\geq h_{\overline{A}}((F_d^{\circ(n+m)}\circ s_d)(\alpha))=d^m h_{\overline{A}}((F_d^{\circ n}\circ s_d)(\alpha))\geq0.
\]
So $h_{\overline{A}}(F_d^{\circ n}(s_d(\alpha)))=0$ for every PCF $\alpha\in\MPoly^d(\olQ)$. After pullback, we obtain a nef adelic line bundle $\overline{A}_Y$ on a normal projective compactification $\overline{Y}\supseteq Y$, with $A_Y$ ample. By Proposition~\ref{prop:orb}, for every Galois-invariant subset $P\subseteq V(\Lab)$ with
\[
\overline{\{a\in P\colon h_{\mathrm{crit}}(a)<\epsilon\}}=V_{\olQ},\quad\forall\epsilon>0,
\]
there is a Galois-invariant subset $Q_P\subseteq Y(\Lab)\cap\varphi^{-1}(P)$ satisfying
\[
G(Q_P)\subseteq Q_P,\quad L(x)=L(\varphi(x)),\quad h_{\overline{A}_Y}(G(x))=d^lh_{\overline{A}_Y}(x)\quad(x\in Q_P),
\]
and $\{x\in Q_P\colon h_{\overline{A}_Y}(x)<\epsilon\}$ is Zariski dense in $Y_{\olQ}$ for all $\epsilon>0$.

The next ingredient is the intertwined equivalence relation $\approx$ on $\MPoly^d(\C)$, defined in \S~\ref{sec:int}. This relation is a constructible subset of $(\MPoly^d)^2$ defined over $\Q$; see Proposition~\ref{prop:rel}(1). There is an integer $M_d$ such that every intertwined class in $\MPoly^d(\C)$ has cardinality $\leq M_d$; see \cite[Theorem~3.46]{FG22} and Proposition~\ref{prop:Md}. By Proposition~\ref{prop:rel}(2), there exists a nonconstant rational function $u_V\in L(V)\setminus L$ such that for every irreducible component $Z\subseteq V\times_L V$ of the Zariski closure of the intertwined relation on $V$ whose projections $p_1,p_2:Z\to V$ are dominant, one has
\[
p_1^*u_V=p_2^*u_V\quad\text{in }L(Z).
\]

The local correspondence step adapts ideas from Ji--Song--Xie \cite[\S~6]{JSX26}. Proposition~\ref{prop:tube} is the form needed here for Galois-invariant sets of small points contained in a fixed closed diagonal or Frobenius tube. The proof of Proposition~\ref{prop:tube} combines Yuan's equidistribution theorem \cite{Yuan2008} with the Chambert-Loir model measure \cite{CL2006} to force the special fiber into the diagonal or Frobenius graph. For a Galois-invariant subset $B\subseteq\overline{Y}(\olQ)^2$, define
\[
B_\epsilon=\left\lbrace(x,y)\in B\colon h_{\overline{A}_Y}(x)+h_{\overline{A}_Y}(y)<\epsilon\right\rbrace\quad(\epsilon>0)\quad\text{and}\quad Z(B)=\bigcap_{n=1}^\infty\overline{B_{1/n}}.
\]
Suppose that $B$ lies in a fixed closed diagonal or Frobenius tube. By Proposition~\ref{prop:tube}, every irreducible component $C$ of $Z(B)$ dominating both factors is bi-finite. Lemma~\ref{lem:base} connects the intertwined relation for parameters with bi-finite correspondences on $\overline{Y}$. The proof of Lemma~\ref{lem:base} uses Proposition~\ref{prop:deg} and the invariant-subvariety decomposition for split polynomial endomorphisms \cite[Proposition~5.3(i)]{GN17} (see also \cite{MS14}).

We continue the proof. By the Chebotarev density theorem, take a rational prime $p>\max\{2,M_d\}$ splitting completely in $L$ and a place $v\in M_L$ above $p$ at which $\overline{A}_Y$ is induced by a relatively ample model with geometrically integral special fiber. Then $L_v=\Q_p$. Fix an extension $\bar{v}$ of $v$ to $\olQ$. Set $\delta=p^{-1/(p-1)}<1$. Define
\begin{align*}
J_{\mathrm{w}}&=\left\lbrace j\in\Z_{>0}\colon L(a_j)_{\bar{v}}/\Q_p\text{ is wildly ramified}\right\rbrace,\\
J_{\mathrm{t}}&=\left\lbrace j\in\Z_{>0}\colon L(a_j)_{\bar{v}}/\Q_p\text{ is tamely ramified}\right\rbrace=\Z_{>0}\setminus J_{\mathrm{w}}.
\end{align*}
Then at least one of $J_{\mathrm{w}}$ and $J_{\mathrm{t}}$ is infinite. Define $P_{\mathrm{w}}=\Gal(\Lab/L)\cdot\{a_j\colon j\in J_{\mathrm{w}}\}$ and $P_{\mathrm{t}}=\Gal(\Lab/L)\cdot\{a_j\colon j\in J_{\mathrm{t}}\}$.

(1) Suppose that $J_{\mathrm{w}}$ is infinite.

For $x\in Q_{P_{\mathrm{w}}}$ with $E=L(x)=L(\varphi(x))$, Lemma~\ref{lem:in}(1) produces an order-$p$ inertia element $\tau_E\in I_{\bar{v}}(E/L)\subseteq\Gal(E/L)$ satisfying
\[
|\tau_E(z)-z|_p\leq\delta\quad(z\in\sO_{E_{\bar{v}}}).
\]
Then
\[
B_{\mathrm{In}}=\left\lbrace(x,\tau_{L(x)}(x))\colon x\in Q_{P_{\mathrm{w}}}\right\rbrace
\]
is contained in the closed diagonal tube $T_1(\delta)$ (see \S~\ref{sec:tube}), with $(G\times G)(B_{\mathrm{In}})\subseteq B_{\mathrm{In}}$. Using the fact that $\Gal(\Lab/L)$ is abelian, we deduce the Galois-invariance of $B_{\mathrm{In}}$. Let $Z_{\mathrm{In}}=Z(B_{\mathrm{In}})$, which is Zariski closed in $\overline{Y}_{\olQ}^2$. Take a periodic cycle of components of $Z_{\mathrm{In}}$ that dominate both factors, and let $C_{\mathrm{In}}$ be their reduced union. Applying Proposition~\ref{prop:tube}, we see that $C_{\mathrm{In}}$ is a bi-finite correspondence on $\overline{Y}_{\olQ}$ satisfying
\[
(G\times G)(C_{\mathrm{In}}\cap Y^2_{\olQ})\subseteq C_{\mathrm{In}}\cap Y^2_{\olQ}.
\]
Using Lemma~\ref{lem:base}(3), we obtain a pair $(x,\tau_{L(x)}(x))\in C_{\mathrm{In}}$ such that
\[
\varphi(x),\tau_{L(x)}(\varphi(x)),\ldots,\tau_{L(x)}^{p-1}(\varphi(x))
\]
are $p$ distinct points in one intertwined class, contradicting $p>M_d$.

(2) Suppose that $J_{\mathrm{t}}$ is infinite.

Choose $\sigma$ in the decomposition group of $\olQ/L$ at $\bar{v}$ whose action on the residue field $\overline{\F}_p$ is $z\mapsto z^p$. As in (1), the set
\[
B_{\mathrm{Fr}}=\left\lbrace(x,\sigma(x))\colon x\in Q_{P_{\mathrm{t}}}\right\rbrace
\]
is Galois-invariant, with $(G\times G)(B_{\mathrm{Fr}})\subseteq B_{\mathrm{Fr}}$. Using Lemma~\ref{lem:in}(2), we obtain $B_{\mathrm{Fr}}\subseteq T_p(\delta)$, the closed tube around the Frobenius graph. As in (1), we construct a bi-finite correspondence $C_{\mathrm{Fr}}$ on $\overline{Y}_{\olQ}$. By the Frobenius case of Proposition~\ref{prop:tube}, for every irreducible component $C_i$ of $C_{\mathrm{Fr}}$ with normalization $W$ and the two projections $\pi_j:W\to C_i\hookrightarrow\overline{Y}^2_{\olQ}\to\overline{Y}_{\olQ}$ ($1\leq j\leq 2$), one has
\begin{equation}\label{eq:tame1}
\pi_2^*A_{Y,\olQ}\equiv p\,\pi_1^*A_{Y,\olQ}.
\end{equation}
On the other hand, Lemma~\ref{lem:base}(2) shows
\begin{equation}\label{eq:tame2}
\pi_1^*(U)=\pi_2^*(U),
\end{equation}
where $U=\varphi^*(u_V)\in L(\overline{Y})\setminus L$. Lemma~\ref{lem:inv} shows that \eqref{eq:tame1} and \eqref{eq:tame2} cannot hold simultaneously. We obtain a contradiction.

Both the wild and tame cases are impossible, completing the proof of Theorem~\ref{thm:main}.

\subsection{The cyclotomic case}\label{sec:cyc}
Let $\mu_\infty$ be the group of all roots of unity in $\C$, and let
\[
\Kc=K(\mu_\infty)\subseteq\Kab
\]
be the maximal cyclotomic extension of a number field $K$ in $\olQ$. Note that $\Q^{\mathrm{c}}=\Q^{\mathrm{ab}}$ by the Kronecker--Weber theorem \cite[Chapter~\Rom{5}, Theorem~1.10]{Neukirch1999}.

Clearly, Theorem~\ref{thm:pcf} implies the following finiteness result:

\begin{theorem}\label{thm:cyc}
For every number field $K$ and every integer $d\geq2$, the set
\[
\sP_d(\Kc):=\left\lbrace[f]\in\MPoly^d(\Kc)\colon f\text{ is PCF}\right\rbrace
\]
is finite.
\end{theorem}

The cyclotomic integral-point rigidity theorem \cite[Theorem~1.2]{JXZ2026} also gives a different proof of Theorem~\ref{thm:cyc}, which we sketch below.

Fix a closed immersion $\iota:X_d\hookrightarrow\A^N_\Q$. There are $D_{d,\iota}\in\Z_{>0}$ and $B_{d,\iota}\geq1$ such that every coordinate $\xi$ of
\[
\iota\left(F_d^{\circ n}(s_d(\alpha))\right)\quad(\text{PCF }\alpha\in\MPoly^d(\olQ),\ n\geq0)
\]
satisfies
\begin{equation}\label{eq:bd}
D_{d,\iota}\xi\in\sO_{\olQ}\quad\text{and}\quad C(\xi)\leq B_{d,\iota}.
\end{equation}
Here for $x\in\olQ$, its \emph{house} is
\[
C(x)=\max_{\tau:\Q(x)\hookrightarrow\C}|\tau(x)|.
\]
These estimates follow from uniform local bounds for monic centered PCF polynomials and critical orbits.

Suppose that $\#\sP_d(\Kc)=\infty$ for some number field $K$ and $d\geq2$. After a finite extension $L/K$ contained in $\Kc$, take a positive-dimensional geometrically integral subvariety $V\subseteq\MPoly^d_L$ such that
\[
P=\{\alpha\in V(\Lc)\colon\alpha\text{ is PCF}\}
\]
is Zariski dense in $V_{\olQ}$. Note that $\Lc=\Kc$. As in \S~\ref{sec:skt}, restrict the family of Proposition~\ref{prop:sys} to $X_V:=\left(X_d\times_{\MPoly^d}V\right)_{\mathrm{red}}$. Applying Lemma~\ref{lem:orb} and Thurston rigidity as in \S~\ref{sec:skt}, there are $l\geq1$ and an irreducible $Y\subseteq X_d$ satisfying
\[
p_d(Y)=V,\quad s:=\dim Y-\dim V>0,\quad\text{and}\quad F_d^{\circ l}(Y)=Y.
\]
Set $G=F_d^{\circ l}|_Y$ and $\varphi=p_d|_Y$. Using \eqref{eq:bd}, we obtain a Zariski-dense set
\[
Q\subseteq Y(\Lc)\cap \varphi^{-1}(P)
\]
satisfying $G(Q)\subseteq Q$ and the same common-denominator and house bounds. After a finite extension inside $\Lc$, assume that $Y/L$ is geometrically integral.

Applying the cyclotomic integral-point rigidity theorem \cite[Theorem~1.2]{JXZ2026}, we conclude that $(Y,G)$ is of \emph{monomial type} in the sense of \cite[Definition~1.1]{JXZ2026}: there are an integer $m\geq1$, a linear torus $T=\mathbb{G}_m^t$, an algebraic group endomorphism $g:T\to T$, and a dominant morphism $\Phi:T\to Y_{\olQ}$ such that
\[
G^{\circ m}\circ\Phi=\Phi\circ g.
\]
Replace $T$ by $g^{\circ n}(T)$ for $n$ sufficiently large. Then the restricted presentation still dominates $Y$, and $g$ is now an isogeny.

By \cite[Lemma~2.6]{JXZ2026}, after passing to a finite isogeny cover and lifting $g$, we may write
\[
T=T_0\times T_+\quad\text{and}\quad g=(g_0,g_+),
\]
where $g_0:T_0\to T_0$ and $g_+:T_+\to T_+$ are group endomorphisms of linear tori, every eigenvalue of (the exponent matrix of) $g_0$ is a root of unity, and $g_+$ has no root-of-unity eigenvalue. Torsion points are Zariski dense in $T_0$, and every torsion point is periodic under $g_0$. The absence of root-of-unity eigenvalues implies that
\[
\olQ(T_+)^{(g_+^{\circ q})^*}=\olQ\quad (q\geq1).
\]
Combining this with $\varphi\circ G=\varphi$, we deduce that
\[
\varphi\circ\Phi=h_0\circ\pr_0
\]
for a dominant morphism $h_0:T_0\to V$, where $\pr_0:T\to T_0$ is the first projection.

By Proposition~\ref{prop:deg}, we have $\deg G=d^{sl}>1$. Using this fact, we can show that there is a nonempty open subset of $T_0$ on which the slices
\[
\Phi_u:T_+\to\varphi^{-1}(h_0(u))
\]
are nonconstant. Since $h_0$ is dominant and there are only finitely many exceptional polynomial classes, take $u\in T_0(\olQ)_{\mathrm{tors}}$ in the nonconstant-slice locus with $v_0=h_0(u)$ non-exceptional; see \S~\ref{sec:int}. Take a monic centered representative $f$ of $v_0$. Let $q\geq1$ be the period of $u$ under $g_0$, and set $M=lmq$. Restricting the equivariance identity to $\{u\}\times T_+$ yields
the semiconjugacy:
\[
\Phi_u:T_+\to \Sym^{d-1}(\A^1)/\Gamma_f\quad\text{and}\quad\overline{F}_f^{\circ M}\circ\Phi_u=\Phi_u\circ g_+^{\circ q},
\]
where $\Gamma_f\subseteq\mu_{d-1}$ is the stabilizer of $f$ and $\overline{F}_f$ is induced by $\Sym^{d-1}(f)$. Pull back $\Sym^{d-1}(\A^1)\to\Sym^{d-1}(\A^1)/\Gamma_f$ along $\Phi_u$. Among the normalizations of the components dominating $T_+$, choose a periodic component for the induced finite map. After passing to an iterate, we obtain a finite Galois cover $\pi_1:T_1\to T_+$ and an equivariant nonconstant morphism $\Phi_1:T_1\to\Sym^{d-1}(\A^1)$, with $T_1$ normal. After a further iterate, by normalizing a component of the finite cover parametrizing all orderings of the distinct points in the generic divisor $\Phi_1(\eta_{T_1})$, we construct a finite Galois cover $\pi_2:W\to T_1$ and a finite endomorphism $H:W\to W$, with $W$ normal and $\pi\circ H=g_+^{\circ r}\circ\pi$ for some $r\geq1$, where $\pi=\pi_1\circ\pi_2$. We can prove that both $\pi_1$ and $\pi_2$ are \'etale, and hence $\pi$ is \'etale. Thus, $W$ is a torus after choosing an origin. After translating a fixed point of $H$ to the origin, $H$ is a group endomorphism. One of the ordered point functions $\psi:W\to\A^1$ (obtained in the construction of $\pi_2$) is nonconstant and satisfies
\[
f^{\circ k}\circ\psi=\psi\circ H
\]
for some $k\geq1$. Then $(\A^1,f)$ is of monomial type. By the classification in \cite[\S~1.3]{JXZ2026}, $f$ is exceptional, a contradiction.

\subsection{Rational map conjectures}
Fix an integer $d\geq2$. Let $\Rat_d$ be the space of endomorphisms of $\P^1$ of degree $d$, which is a smooth affine variety of dimension $2d+1$; see \cite{Silverman2012}. The group $\PGL_2=\operatorname{Aut}(\P^1)$ acts on $\Rat_d$ by conjugacy. The geometric quotient
\[\sM_d:=\Rat_d/\PGL_2=\Spec\left(\sO(\Rat_d)^{\PGL_2}\right)\]
is the (coarse) \emph{moduli space of endomorphisms of degree $d$}, which is an affine variety of dimension $2d-2$; see \cite[Theorem~4.36~(c)]{Silverman2007} and \cite[Chapter~2]{Silverman2012}.

We recall the exceptional rational maps, namely those arising from endomorphisms of elliptic curves or algebraic tori.
\begin{definition}\label{def:excep}
Let $f:\P^1_\C\to\P^1_\C$ be an endomorphism of degree $d\geq2$.
\begin{itemize}
\item The map $f$ is called a \emph{Latt\`es map} if there are an elliptic curve $E$, an endomorphism $A:E\to E$, and a nonconstant morphism $\pi:E\to\P^1_\C$ such that $f\circ\pi=\pi\circ A$. The map $f$ is called \emph{flexible Latt\`es} if there are an elliptic curve $E$, an integer $m\in\Z\setminus\{0,\pm1\}$, and a finite morphism $\pi:E\to\P^1_\C$ of degree $2$ such that $f\circ\pi=\pi\circ[m]$. A non-flexible Latt\`es map is called \emph{rigid Latt\`es}.
\item For rational maps on $\P^1$, we use the following one-dimensional convention: $f$ is of \emph{monomial type} if there are $n\in\{\pm d\}$ and a nonconstant morphism $\pi:\P^1_\C\to\P^1_\C$ such that $f\circ\pi=\pi\circ z^n$. It is well known that $f$ is of monomial type if and only if it is conjugate to $z^{\pm d}$ or $\pm T_d(z)$, where $T_d(z)$ is the \emph{(normalized) Chebyshev polynomial of degree $d$}, i.e., the unique monic polynomial $T_d\in\C[z]$ of degree $d$ such that $T_d(z+z^{-1})=z^d+z^{-d}$. See \cite[Chapter~6]{Silverman2007}, \cite{Milnor2006}, and \cite[\S~1.3]{JXZ2026}.
\item $f$ is called \emph{exceptional} if it is Latt\`es or of monomial type.
\end{itemize}
The above notions depend only on the conjugacy class of $f$. Moreover, it is well known that $f$ is exceptional if and only if $f^{\circ k}$ is exceptional for some (equivalently, every) $k\in\Z_{>0}$.
\end{definition}
All exceptional maps are PCF.

Let $\FL_d\subseteq\Rat_d$ and $[\FL_d]\subseteq\sM_d$ be the flexible Latt\`es loci, which are Zariski closed. The locus $[\FL_d]$ is nonempty if and only if $d$ is a perfect square in $\Z$, and in this case $[\FL_d]$ is of pure dimension $1$.

\medskip

Since the critical height is invariant under conjugacy, we obtain the critical height function on $\sM_d$:
\[
h_{\mathrm{crit}}:\sM_d(\olQ)\to\R_{\geq0},\,[f]\mapsto h_{\mathrm{crit}}(f).
\]
Since there are only finitely many rigid Latt\`es conjugacy classes in $\sM_d(\C)$ (cf.~\cite[\S~5]{Milnor2006}), by Ingram \cite{Ingram2018} we conclude that the critical height $h_{\mathrm{crit}}$ on $\sM_d$ is comparable to an ample moduli height away from $[\FL_d]$. As in the polynomial case, Northcott's theorem then implies finiteness for PCF map classes in $\sM_d(\olQ)\setminus[\FL_d]$ of bounded degree $[\Q([f]):\Q]\leq B$.

Motivated by Theorems \ref{thm:main}, \ref{thm:pcf}, and \ref{thm:cyc}, we propose the following conjectures.
\begin{conjecture}\label{conj:rbog}
For every number field $K$ and every integer $d\geq2$, there exists a constant $\epsilon_{K,d}>0$ such that the set
\[
\left\lbrace\alpha\in\sM_d(\Kab)\setminus[\FL_d]\colon h_{\mathrm{crit}}(\alpha)<\epsilon_{K,d}\right\rbrace=\left\lbrace\alpha\in\sM_d(\Kab)\setminus[\FL_d]\colon h_{\mathrm{crit}}(\alpha)=0\right\rbrace
\]
is finite.
\end{conjecture}
\begin{conjecture}\label{conj:rab}
For every number field $K$ and every integer $d\geq2$, the set
\[
\sR_d(\Kab):=\left\lbrace[f]\in\sM_d(\Kab)\setminus[\FL_d]\colon f\text{ is PCF}\right\rbrace
\]
is finite.
\end{conjecture}
\begin{conjecture}\label{conj:rcyc}
For every number field $K$ and every integer $d\geq2$, the set
\[
\sR_d(\Kc):=\left\lbrace[f]\in\sM_d(\Kc)\setminus[\FL_d]\colon f\text{ is PCF}\right\rbrace
\]
is finite.
\end{conjecture}
Clearly, one has:
\begin{center}
Conjecture~\ref{conj:rbog}\ $\Longrightarrow$\ Conjecture~\ref{conj:rab}\ $\Longrightarrow$\ Conjecture~\ref{conj:rcyc}.
\end{center}
For $K=\Q$, Conjecture~\ref{conj:rab} and Conjecture~\ref{conj:rcyc} coincide.

When $d$ is a square, there are infinitely many flexible Latt\`es classes of degree $d$ admitting representatives over $\Q$; all of these classes are PCF. Thus, it is necessary to exclude $[\FL_d]$ in all three conjectures.

\subsection*{Organization}
In \S~\ref{sec:poly}, we construct the universal critical-divisor family, a projective compactification with a nef adelic line bundle for which orbit-height formulas hold, and prove the relative degree formula. In \S~\ref{sec:intcorr}, we study intertwined polynomials and small-point correspondences. In \S~\ref{sec:proof}, we finish the proof of Theorem~\ref{thm:main}. In \S~\ref{sec:app}, we prove Theorem~\ref{thm:pair}.

\section{The universal critical-divisor family}\label{sec:poly}
Throughout this section, let $d\geq2$ be an integer.
\subsection{The universal family}\label{sec:2.1}
Let
\[
\Poly^d=\Spec\Q[a_0,\ldots,a_d,a_d^{-1}]
\]
be the parameter space of polynomials $a_dz^d+\cdots+a_0$ of degree $d$. The group $\operatorname{Aff}$ consists of degree-one polynomials $z\mapsto uz+v$ with $u\ne0$, and acts on $\Poly^d$ by conjugacy. Then $\MPoly^d=\Poly^d/\operatorname{Aff}$ is the geometric quotient. Every degree-$d$ polynomial over $\C$ is affinely conjugate to some monic centered polynomial
\[
f_{\ba}(z)=z^d+\sum_{j=0}^{d-2}a_jz^j,\quad\ba=(a_0,\ldots,a_{d-2})\in\C^{d-1}.
\]
Let $A_d=\Spec\Q[a_0,\ldots,a_{d-2}]\subseteq\Poly^d$ be the parameter space of monic centered polynomials of degree $d$. Over $\olQ$, the morphism $A_d\subseteq\Poly^d\to\MPoly^d$ is realized as the finite quotient
\[
A_d\to A_d/\mu_{d-1}=\MPoly^d,
\]
where $\mu_{d-1}$ is the group scheme of $(d-1)$-st roots of unity, and it acts on $A_d$ by
\[
\zeta\cdot\ba=\left(a_j\zeta^{j-1}\right)_{0\leq j\leq d-2}.
\]
Then $f_{\zeta\cdot\ba}(z)=\zeta^{-1}f_{\ba}(\zeta z)$. See \cite[\S~2.1]{FG22}.

Let $S_{d-1}=\Sym^{d-1}(\A^1)\simeq\A^{d-1}$, whose coordinates are the elementary symmetric polynomials. Geometric points of $S_{d-1}$ are written as effective divisors of the form $[z_1]+\cdots+[z_{d-1}]$. For $\lambda\in\C$ and $D=\sum_{i=1}^{d-1}[z_i]\in S_{d-1}(\C)$, set
\[
\lambda D:=\sum_{i=1}^{d-1}[\lambda z_i]\in S_{d-1}(\C).
\]

We construct the universal critical-divisor family $(X_d,F_d,p_d,s_d)$ over $\MPoly^d$.
\begin{proposition}\label{prop:sys}
There are an affine variety $X_d$, a finite endomorphism $F_d:X_d\to X_d$, a morphism $p_d:X_d\to\MPoly^d$ with $p_d\circ F_d=p_d$, and a section $s_d:\MPoly^d\to X_d$ of $p_d$, all defined over $\Q$, such that for every $\alpha\in\MPoly^d(\C)$,
\begin{equation}\label{eq:pcf}
\alpha\text{ is PCF}\quad\Longleftrightarrow\quad s_d(\alpha)\in X_d(\C)\text{ is }F_d\text{-preperiodic}.
\end{equation}
\end{proposition}
\begin{proof}
On $\widetilde{X}_d:=A_d\times S_{d-1}$, define
\[
\widetilde{F}_d(\ba,[z_1]+\cdots+[z_{d-1}])=(\ba,[f_{\ba}(z_1)]+\cdots+[f_{\ba}(z_{d-1})]).
\]
Note that $\widetilde{F}_d:\widetilde{X}_d\to\widetilde{X}_d$ is an endomorphism defined over $\Q$. On the ordered cover $A_d\times(\A^1)^{d-1}$, the endomorphism
\[
A_d\times(\A^1)^{d-1}\to A_d\times(\A^1)^{d-1},\,(\ba,z_1,\ldots,z_{d-1})\mapsto(\ba,f_{\ba}(z_1),\ldots,f_{\ba}(z_{d-1}))
\]
is clearly finite. The induced extension of invariant coordinate rings is integral and of finite type, hence finite. Thus, $\widetilde{F}_d$ is finite.

Let $c_1,\ldots,c_{d-1}$ be the roots of $f_{\ba}'$, counted with multiplicity. For every $n\geq0$, define
\[
D_n(\ba):=\sum_{i=1}^{d-1}[f_{\ba}^{\circ n}(c_i)]\in S_{d-1}.
\]
In particular, $D_0(\ba)$ is the critical divisor of $f_{\ba}$. Note that $\widetilde{s}_d(\ba)=(\ba,D_0(\ba))$ defines a section $\widetilde{s}_d:A_d\to\widetilde{X}_d$ of the projection $\widetilde{X}_d\to A_d$ over $\Q$. The finite group scheme
\[
\mu_{d-1}=\Spec\Q[t]/(t^{d-1}-1),
\]
acts on $\widetilde{X}_d$ by 
\[
\zeta\cdot(\ba,x)=(\zeta\cdot\ba,\zeta^{-1}x).
\]
Then the projection $\widetilde{X}_d\to A_d$, $\widetilde{F}_d$, and $\widetilde{s}_d$ are $\mu_{d-1}$-equivariant. So they descend to $X_d:=\widetilde{X}_d/\mu_{d-1}$, giving
\[
p_d:X_d\to\MPoly^d,\quad F_d:X_d\to X_d,\quad\text{and}\quad s_d:\MPoly^d\to X_d.
\]
The descended endomorphism $F_d$ is also finite.

For $n\geq1$, observe that $\widetilde{F}_d^{\circ n}(\widetilde{s}_d(\ba))=(\ba,D_n(\ba))$. It is straightforward to verify that for every $\ba\in A_d(\C)$, $f_\ba$ is PCF if and only if $\widetilde{s}_d(\ba)$ is $\widetilde{F}_d$-preperiodic, and that $s_d([\ba])$ is $F_d$-preperiodic in this case. Suppose that $s_d([\ba])$ is $F_d$-preperiodic. Then
\[
\zeta\cdot\ba=\ba\quad\text{and}\quad D_m(\ba)=\zeta^{-1}D_n(\ba)
\]
for some $\zeta\in\mu_{d-1}(\olQ)$ and some integers $m>n\geq0$. Equivariance of $\widetilde{F}_d$ implies
\[
D_{n+k(m-n)}(\ba)=\zeta^{-k}D_n(\ba)
\]
for every $k\geq1$. For $k_0=\operatorname{ord}(\zeta)$, we conclude that $D_{n+k_0(m-n)}(\ba)=D_n(\ba)$, and hence
\[
\widetilde{F}_d^{\circ(n+k_0(m-n))}(\widetilde{s}_d(\ba))=\widetilde{F}_d^{\circ n}(\widetilde{s}_d(\ba)),
\]
which completes the proof of \eqref{eq:pcf}.
\end{proof}
From now on, we fix the tuple $(X_d,F_d,p_d,s_d)$ constructed in the proof of Proposition~\ref{prop:sys}.

\subsection{A nef adelic line bundle}
Let $F$ be a number field. For $v\in M_F$, write $n_v=[F_v:\Q_v]/[F:\Q]$. For a polynomial $f\in F[z]$ of degree $d\geq2$ and $x\in\C_v$, define
\[
\hat{\lambda}_{f,v}(x)=\lim_{n\to\infty}\frac{1}{d^n}\log\max\{1,|f^{\circ n}(x)|_v\}.
\]
Then $\hat{\lambda}_{f,v}\geq0$ and $\hat{\lambda}_{f,v}(f(x))=d\hat{\lambda}_{f,v}(x)$. Moreover, $\hat{h}_f(x)=\sum_{v\in M_F}n_v\hat{\lambda}_{f,v}(x)$ for $x\in F$; see \cite[Theorem~3.29]{Silverman2007}. Suppose that $\Crit_{\mathrm{fin}}(f)\subseteq F$. Define
\[
h_{\mathrm{crit}}^{\max}(f)=\sum_{v\in M_F}n_v\max_{c\in\Crit_{\mathrm{fin}}(f)}\hat{\lambda}_{f,v}(c).
\]
For a linear polynomial $\ell(z)=az+b$ with $a\in F^\times$ and $b\in F$, it is clear that
\begin{equation}\label{eq:affgreen}
\hat{\lambda}_{\ell\circ f\circ\ell^{-1},v}(\ell(x))=\hat{\lambda}_{f,v}(x)\quad(v\in M_F,\ x\in\C_v).
\end{equation}
Clearly, $h_{\mathrm{crit}}^{\max}(f)$ is independent of the choice of $F$, and it is invariant under affine conjugacy and Galois action. Thus, the sum of the local maxima defines a Galois-invariant function
\[
h_{\mathrm{crit}}^{\max}:\MPoly^d(\olQ)\to\R_{\geq0}.
\]
In general, this sum of local maxima is not the maximum of the global canonical heights of the critical points. Since a degree-$d$ polynomial has exactly $d-1$ finite critical points (counted with multiplicity), for $\alpha\in\MPoly^d(\olQ)$, we have
\begin{equation}\label{eq:cmp}
\frac{h_{\mathrm{crit}}(\alpha)}{d-1}\leq h_{\mathrm{crit}}^{\max}(\alpha)\leq h_{\mathrm{crit}}(\alpha),
\end{equation}
and
\[
\alpha\text{ is PCF}\ \Longleftrightarrow\ h_{\mathrm{crit}}(\alpha)=0\ \Longleftrightarrow\ h_{\mathrm{crit}}^{\max}(\alpha)=0.
\]

We construct a nef adelic line bundle on a compactification of the universal critical-divisor family, and relate its height to $h_{\mathrm{crit}}^{\max}$.
\begin{theorem}\label{thm:met}
After a finite extension $L/\Q$, there are a normal projective compactification $X_{d,L}\subseteq\overline{X}_d$ and a nef adelic line bundle $\overline{A}$ on $\overline{X}_d$, with $A$ ample, such that
\begin{equation}\label{eq:orbitheight}
h_{\overline{A}}((F_d^{\circ n}\circ s_d)(\alpha))=k_d d^n h_{\mathrm{crit}}^{\max}(\alpha)\quad(\alpha\in\MPoly^d(\olQ),\ n\geq0),
\end{equation}
where $k_d=(d-1)\operatorname{lcm}(1,\ldots,d)\in\Z_{>0}$. In particular, $h_{\overline{A}}(F_d^{\circ n}(s_d(\alpha)))=0$ for all PCF $\alpha\in\MPoly^d(\olQ)$ and $n\geq0$.

Moreover, there are a finite subset $S\subseteq M_L$ containing all archimedean places, a projective flat model $\sX_d\to\Spec\sO_{L,S}$ of $\overline{X}_d$, and a relatively ample line bundle $\sA$ extending $A$ such that for every $v\in M_L\setminus S$, the $v$-adic metric of $\overline{A}$ is the model metric of $(\sX_d,\sA)$ at $v$.
\end{theorem}
\begin{proof}
Let $B=\A^{2d-2}$. Write $\bc=(c_1,\ldots,c_{d-1})$ and $\bz=(z_1,\ldots,z_{d-1})$. For $\bc=(c_1,\ldots,c_{d-1})\in\A^{d-1}(\olQ)$, define
\[
g_{\bc}(z)=\int_0^z\prod_{i=1}^{d-1}(t-c_i)\dd t.
\]
Then $g_{\bc}(0)=0$, the leading coefficient of $g_{\bc}$ is $1/d$, and the finite critical points of $g_{\bc}$ counted with multiplicity are exactly $c_1,\ldots,c_{d-1}$.

\medskip

\noindent\textit{The metric on the ordered cover.} For $n\geq1$, define an endomorphism $\Theta_n:B\to B$ by
\[
\Theta_n(\bc,\bz)=\left(g_{\bc}^{\circ n}(c_1),\ldots,g_{\bc}^{\circ n}(c_{d-1}),g_{\bc}^{\circ n}(z_1),\ldots,g_{\bc}^{\circ n}(z_{d-1})\right).
\]
Observe that every coordinate of $\Theta_n$ is a polynomial in $(\bc,\bz)$ of total degree $d^n$, whose homogeneous part of maximal degree $d^n$ is a nonzero rational scalar multiple of the $d^{n-1}$-th power of the corresponding coordinate of $\Theta_1$. On $\P^{2d-2}$, homogenize all coordinates of $\Theta_n$ to degree $d^n$ and add the coordinate $t^{d^n}$. A common zero of these homogeneous coordinates must satisfy $t=0$, and
\[
g_{\bc}(c_i)=0\ (1\leq i\leq d-1),\quad\text{and}\quad g_{\bc}(z_j)=0\ (1\leq j\leq d-1).
\]
By \cite[Lemma~2.4]{Ingram2012}, the first set of equalities implies $\bc=\mathbf{0}$ and $g_{\mathbf{0}}(z)=z^d/d$; so the vanishing of $g_{\mathbf{0}}(z_j)$ for all $j$ implies $\bz=0$. Consequently, $\Theta_n$ extends to an endomorphism
\[
\Theta_n:\P^{2d-2}\to\P^{2d-2}
\]
of algebraic degree $d^n$, defined over $\Q$.

Let $\overline{\sO(1)}_{\mathrm{can}}$ be the line bundle $\sO(1)$ on $\P^{2d-2}_\Q$ with the canonical metric defined below. For $v\in M_\Q$, a linear form $s$, and a point $x=[x_0:\cdots:x_{2d-2}]\in\P^{2d-2}(\C_v)$, the canonical metric is given by
\[
\Vert s(x)\Vert_{\mathrm{can},v}=\frac{|s(x_0,\dots,x_{2d-2})|_v}{\max\left\lbrace|x_0|_v,\ldots,|x_{2d-2}|_v\right\rbrace}.
\]
Then $\overline{\sO(1)}_{\mathrm{can}}$ is a nef adelic line bundle. Let $n\geq1$. Use the homogeneous coordinate $t^{d^n}$ to identify $\Theta_n^*\sO(1)$ with $\sO(d^n)$. Define $\overline{H}_n$ as $\sO(1)$ with the $d^n$-th root of the pullback metric; in additive notation,
\[
\overline{H}_n=\frac{1}{d^n}\Theta_n^*\overline{\sO(1)}_{\mathrm{can}}.
\]
Then $\overline{H}_n$ is a nef adelic line bundle on $\P^{2d-2}_\Q$.

On $B(\C_v)$, define the local potential
\[
\psi_{n,v}(\bc,\bz)=\frac{1}{d^n}\log\max\{1,|g_{\bc}^{\circ n}(c_i)|_v,|g_{\bc}^{\circ n}(z_j)|_v\colon1\leq i,j\leq d-1\}.
\]
We claim that for every $v\in M_\Q$, there is a constant $C_{d,v}\geq0$ such that
\begin{equation}\label{eq:psin}
\|\psi_{n+1,v}-\psi_{n,v}\|_{\sup}\leq\frac{C_{d,v}}{d^{n+1}}\quad\text{on }B(\C_v)
\end{equation}
for all $n\geq1$, and that $C_{d,v}=0$ for all but finitely many $v\in M_\Q$.

Fix a place $v$. For arbitrary $\bc=(c_1,\ldots,c_{d-1})$ and $\bz=(z_1,\ldots,z_{d-1})$ in $\C_v^{d-1}$, let $g=g_{\bc}$ and $\rho=\max\{1,|c_1|_v,\ldots,|c_{d-1}|_v\}$. For $n\geq1$, define
\[
M_n=\max\{1,|g^{\circ n}(c_i)|_v,|g^{\circ n}(z_j)|_v\colon1\leq i,j\leq d-1\}.
\]
Write $g_{\bc}(z)=z^d/d+\sum_{j=1}^{d-1}A_j(\bc)z^j$, where $A_j\in\Q[\bc]$ is homogeneous of degree $d-j$. Then we can choose constants $a_{j,d,v}\geq0$ such that
\[
|A_j(\bc)|_v\leq a_{j,d,v}\rho^{d-j}
\]
for $1\leq j\leq d-1$. By the triangle inequality, there is $K_{d,v}\geq1$ such that
\begin{equation}\label{eq:upp}
\max\{1,|g_{\bc}(w)|_v\}\leq K_{d,v}\max\{\rho,|w|_v\}^d\quad\text{for every }w\in\C_v,
\end{equation}
and
\begin{equation}\label{eq:low}
|g_{\bc}(w)|_v\geq K_{d,v}^{-1}|w|_v^d>|w|_v\quad\text{if }|w|_v>K_{d,v}\rho.
\end{equation}
Moreover, we may choose $K_{d,v}=1$ for all but finitely many $v\in M_\Q$. By \cite[Lemma~2.4]{Ingram2012}, the homogeneous forms $g_{\bc}(c_i)$ have degree $d$ and no common nonzero zero over $\olQ$. Applying the homogeneous Nullstellensatz, we obtain an integer $k\geq d$ and homogeneous polynomials $H_{ij}\in\Q[\bc]$ of degree $k-d$ such that
\begin{equation}\label{eq:homo}
c_i^k=\sum_{j=1}^{d-1}H_{ij}(\bc)g_{\bc}(c_j),\quad 1\leq i\leq d-1.
\end{equation}
If $\rho>1$, by \eqref{eq:homo} for an index $i_0$ with $|c_{i_0}|_v=\rho$ and the triangle inequality, there is a constant $K'_{d,v}\geq1$ such that
\[
\rho^k\leq K_{d,v}'\rho^{k-d}\max_j|g_{\bc}(c_j)|_v.
\]
We can choose $K_{d,v}'=1$ for all but finitely many $v\in M_\Q$. After increasing $K_{d,v}$, we may assume that
\[
\max_i|g_{\bc}(c_i)|_v\geq K_{d,v}^{-1}\rho^d\quad \text{whenever }\rho>1,
\]
and we still have that $K_{d,v}=1$ for all but finitely many $v\in M_\Q$. Let $R_{d,v}=K_{d,v}^{2/(d-1)}\in\R_{\geq1}$. Observe that $R_{d,v}=1$ if $K_{d,v}=1$. Suppose first that $\rho>R_{d,v}$. Then for some $1\leq i_0\leq d-1$, we have $|g(c_{i_0})|_v\geq K_{d,v}^{-1}\rho^d>K_{d,v}\rho$. Applying \eqref{eq:low} inductively, we obtain
\[
|g^{\circ n}(c_{i_0})|_v>K_{d,v}\rho
\]
for all $n\geq1$. Then $M_n>K_{d,v}\rho$ for all $n\geq1$. By \eqref{eq:upp} and \eqref{eq:low}, we conclude that
\[
K_{d,v}^{-1}M_n^d\leq M_{n+1}\leq K_{d,v}M_n^d\quad\text{and}\quad|\log M_{n+1}-d\log M_n|\leq\log K_{d,v}
\]
for every $n\geq1$. Suppose now that $\rho\leq R_{d,v}$. If $M_n>K_{d,v}R_{d,v}\geq K_{d,v}\rho$, the same argument applies. Otherwise, $1\leq M_n\leq K_{d,v}R_{d,v}$, and \eqref{eq:upp} shows
\[
1\leq M_{n+1}\leq K_{d,v}(K_{d,v}R_{d,v})^d.
\]
It follows that
\[
|\log M_{n+1}-d\log M_n|\leq(d+1)\log K_{d,v}+d\log R_{d,v}.
\]
Therefore, setting
\[
C_{d,v}=(d+1)\log K_{d,v}+d\log R_{d,v},
\]
on $B(\C_v)$ we obtain
\[
\|\psi_{n+1,v}-\psi_{n,v}\|_{\sup}\leq\frac{1}{d^{n+1}}\sup_{\bc,\bz\in\C_v^{d-1}}|\log M_{n+1}-d\log M_n|\leq\frac{C_{d,v}}{d^{n+1}}\quad(n\geq1).
\]
By the construction, $C_{d,v}=0$ for all but finitely many $v\in M_\Q$.

For $m,n\geq1$, the difference $\psi_{m,v}-\psi_{n,v}$ is the logarithmic metric ratio of $\overline{H}_m$ and $\overline{H}_n$ on $B_{\C_v}$. By continuity, the bound in \eqref{eq:psin} holds on the projective analytification $\left(\P^{2d-2}_{\C_v}\right)^{\mathrm{an}}$ for the logarithmic metric ratio of $\overline{H}_{n+1}$ and $\overline{H}_n$. Using the fact that $C_{d,v}=0$ for all but finitely many $v\in M_\Q$, it is routine to see that $(\overline{H}_n)_{n\geq1}$ satisfies the Cauchy condition of Yuan--Zhang \cite{YZ26}. By \cite[\S~A.5.2]{YZ26}, the sequence $(\overline{H}_n)_{n\geq1}$ converges to a nef adelic line bundle $\overline{H}$ on $\P^{2d-2}_\Q$ with underlying line bundle $\sO(1)$.

Since $h_{\overline{\sO(1)}_{\mathrm{can}}}$ is the standard height $h$ on $\P^{2d-2}(\olQ)$, for every $x\in\P^{2d-2}(\olQ)$ we have
\[
h_{\overline{H}}(x)=\lim_{n\to\infty}\frac{1}{d^n}h(\Theta_n(x)).
\]
If $g_{\bc}$ is PCF and every $z_j$ is $g_{\bc}$-preperiodic, then $h_{\overline{H}}(\bc,\bz)=0$.

\medskip

\noindent\textit{Finite descent to the quotient.} We next descend $\overline{H}$ to a compactification of $X_d$. Let $\lambda\in\R_{>0}\cap\olQ$ be such that $\lambda^{d-1}=d$. Fix a primitive $(d-1)$-th root of unity $\zeta_{d-1}$. Set $L=\Q(\lambda,\zeta_{d-1})$.
Write
\[
m_{\bc}=\frac{c_1+\cdots+c_{d-1}}{d-1}.
\]
The affine change of coordinates $z\mapsto\lambda z+m_{\bc}$ sends $g_{\bc}$ to the monic centered polynomial
\[
\lambda^{-1}(g_{\bc}(\lambda z+m_{\bc})-m_{\bc})\triangleq f_{\ba(\bc)}.
\]
Define a morphism
\[
\Omega:B_L\to A_{d,L}\times S_{d-1,L},\,(\bc,\bz)\mapsto\left(\ba(\bc),\sum_{j=1}^{d-1}[(z_j-m_{\bc})/\lambda]\right).
\]

Define the weighted projective space
\[
P_w=\operatorname{Proj}L[t,b_0,\ldots,b_{d-2},u_1,\ldots,u_{d-1}],\quad\deg t=1,\quad\deg b_j=d-j,\quad\deg u_i=i.
\]
The functions $b_j/t^{d-j}$ and $u_i/t^i$ identify the affine chart $D_+(t)$ with $A_{d,L}\times S_{d-1,L}$. Equip $L[t,\bc,\bz]$ with the standard grading in which every variable has degree $1$, so that its projective spectrum is $\P^{2d-2}_L$. Define
\begin{align*}
B_0(t,\bc)&=\frac{1}{\lambda}\left(g_{\bc}(m_{\bc})-m_{\bc}t^{d-1}\right),\\
B_j(t,\bc)&=\frac{\lambda^{j-1}}{j!}g_{\bc}^{(j)}(m_{\bc})\ && (1\leq j\leq d-2),\\
U_i(\bc,\bz)&=\frac{1}{\lambda^i}e_i(z_1-m_{\bc},\ldots,z_{d-1}-m_{\bc})\ && (1\leq i\leq d-1),
\end{align*}
where $g_{\bc}^{(j)}$ is the $j$-th derivative and $e_i$ is the $i$-th elementary symmetric polynomial in $d-1$ variables. Each $B_j$ is homogeneous of degree $d-j$, and each $U_i$ is homogeneous of degree $i$. Then the assignments $t\mapsto t$, $b_j\mapsto B_j$, and $u_i\mapsto U_i$ define a degree-preserving homomorphism
\[
L[t,b_0,\ldots,b_{d-2},u_1,\ldots,u_{d-1}]\to L[t,\bc,\bz].
\]
It induces a morphism $\Phi_w:\P^{2d-2}_L\setminus V\to P_w$, where $V$ is the common zero locus of $t,B_0,\ldots,B_{d-2},U_1,\ldots,U_{d-1}$. A direct calculation shows that $V(\olQ)=\emptyset$. Thus, we obtain a morphism
\[
\Phi_w:\P^{2d-2}_L\to P_w.
\]
Using the chart $t=1$, we see that $\Phi_w$ extends $\Omega$.

Set
\[e=\operatorname{lcm}(1,\ldots,d)=\frac{k_d}{d-1}.
\]
Since every weight divides the positive integer $e$, the sheaf $\sO_{P_w}(e)$ is an ample line bundle; see \cite[Remarques~1.8 and Proposition~2.3(i)]{Delorme1975}. Via $\Phi_w$, we obtain an isomorphism $\Phi_w^*\sO_{P_w}(e)\simeq\sO_{\P^{2d-2}}(e)$. If a fiber of $\Phi_w$ is positive-dimensional, it contains a projective integral curve on which $\Phi_w^*\sO_{P_w}(e)\simeq\sO_{\P^{2d-2}}(e)$ has degree zero, a contradiction. Thus, $\Phi_w$ is quasi-finite. Since $\Phi_w$ is proper and $P_w$ is projective, $\Phi_w$ is finite and surjective.

Let $G_0=\mu_{d-1}(L)=\mu_{d-1}(\olQ)$. Define the $G_0$-action on $P_w$ by
\[
\zeta^*(t)=t,\quad\zeta^*(b_j)=\zeta^{j-1}b_j\quad(0\leq j\leq d-2),\quad\zeta^*(u_i)=\zeta^{-i}u_i\quad(1\leq i\leq d-1).
\]
On $D_+(t)\simeq A_{d,L}\times S_{d-1,L}$, this is exactly the action used to construct $X_d$. Define
\[
\overline{X}_d=\operatorname{Proj}L[t,b_0,\ldots,b_{d-2},u_1,\ldots,u_{d-1}]^{G_0},
\]
and let $\pi_w:P_w\to\overline{X}_d$ be the quotient morphism. Then $\pi_w$ is a finite surjective morphism and $\overline{X}_d=P_w/G_0$; see \cite[Chapter~1]{MFK1994}. The invariant ring is a finitely generated normal graded domain, so $\overline{X}_d$ is normal and projective. Since $t$ is $G_0$-invariant, we obtain
\[
X_{d,L}=\left(A_{d,L}\times S_{d-1,L}\right)/G_0\simeq D_+(t)/G_0\subseteq\overline{X}_d.
\]
Thus, $\overline{X}_d$ is a projective compactification of $X_{d,L}$.

The graded action defines a $G_0$-linearization of $\sO_{P_w}(e)$. For every $x\in P_w(\olQ)$, its stabilizer acts on the one-dimensional fiber $\sO_{P_w}(e)(x)$ through a character whose order divides $d-1$. Consequently, we obtain a line bundle
\[
A=\left((\pi_w)_*\sO_{P_w}(k_d)\right)^{G_0}
\]
on $\overline{X}_d$ and a $G_0$-equivariant isomorphism $\pi_w^*A\simeq\sO_{P_w}(k_d)$; see \cite[Chapter~1]{MFK1994}. Note that $A$ is ample because $\pi_w^*A$ is ample. Set $\phi=\pi_w\circ\Phi_w:\P^{2d-2}_L\to\overline{X}_d$, which is finite and surjective. Observe that $\phi^{-1}(X_{d,L})=B_L$. Fix the isomorphism
\[
\eta:\phi^*A\simeq\sO_{\P^{2d-2}_L}(k_d)
\]
obtained from the identifications above.

The morphism $\phi$ identifies $F:=L(\overline{X}_d)$ with a subfield of $E:=L(\P^{2d-2}_L)$, and $E/F$ is finite and separable. Fix an algebraic closure $\overline{E}$ of $E$, and let $\Lambda/F$ be the normal closure of $E/F$ in $\overline{E}$. Let $R$ be the normalization of $\overline{X}_d$ in $\Lambda$. Then $R$ is normal and projective, its structural morphism $q:R\to\overline{X}_d$ is finite and surjective, and $R/\Gamma=\overline{X}_d$, where $\Gamma=\Gal(\Lambda/F)$. Since $\P^{2d-2}_L$ is normal and finite over $\overline{X}_d$, there is a finite surjective morphism $\theta:R\to\P^{2d-2}_L$ satisfying $q=\phi\circ\theta$.

The isomorphism $\theta^*\eta:q^*A\simeq\theta^*\sO_{\P^{2d-2}_L}(k_d)$ allows us to regard $\overline{N}:=\theta^*(\overline{H}^{\otimes k_d})$ as a nef adelic line bundle on $R$ with underlying line bundle $q^*A$. Here we also view $\overline{H}$ as an adelic line bundle on $\P^{2d-2}_L$. For a place $v$ of $L$, denote the metric of $\overline{N}$ at $v$ by $\|\cdot\|_{0,v}$. Since $q\circ\gamma=q$ for every $\gamma\in\Gamma$, each $\gamma^*\overline{N}$ also has underlying line bundle $q^*A$ and is nef. Define a metric on $q^*A$ by
\[
\|(q^*s)(y)\|_{\mathrm{av},v}=\prod_{\gamma\in\Gamma}\|(q^*s)(\gamma(y))\|_{0,v}^{1/\#\Gamma},
\]
where $s$ is a nonvanishing section of $A$ on an open subset $U\subseteq\overline{X}_d$, and $y\in(q^{-1}(U)_{\C_v})^{\mathrm{an}}$. The metric $\|\cdot\|_{\mathrm{av}}=(\|\cdot\|_{\mathrm{av},v})_v$ on $q^*A$ is $\Gamma$-invariant. Since $q^{\mathrm{an}}$ is the quotient by $\Gamma$, we obtain a unique continuous metric $\|\cdot\|_A=(\|\cdot\|_{A,v})_v$ on the line bundle $A$, characterized by
\[
\|s(q(y))\|_{A,v}=\|(q^*s)(y)\|_{\mathrm{av},v}.
\]

Since $\overline{X}_d$ is normal, the algebraic norm for the finite morphism $q$ is defined and $\operatorname{Nm}_q(q^*A)\simeq A^{\otimes\#\Gamma}$; see \cite[Theorem~2.7 and Proposition~2.8]{Li24}. The descended metric $\|\cdot\|_A$ on $A$ is the $\#\Gamma$-th root of the norm metric induced by $\|\cdot\|_{\mathrm{av}}$ on $q^*A$. Its pullback to $q^*A$ is $\|\cdot\|_{\mathrm{av}}$. Using the definition of nefness through model approximation in \cite[\S~A.5.2]{YZ26}, it is routine to verify that $\overline{A}=(A,\|\cdot\|_A)$ is a nef adelic line bundle on $\overline{X}_d$.

\medskip

\noindent\textit{The height formula and good models.} Let $\alpha\in\MPoly^d(\olQ)$ and $n\geq0$. Set
\[
x=(F_d^{\circ n}\circ s_d)(\alpha)\in X_d(\olQ).
\]
Since $q$ is finite and surjective, we can choose $y\in R(\olQ)$ with $q(y)=x$. Let $\gamma\in\Gamma$ be arbitrary. Then we have $\phi\circ\theta(\gamma(y))=x$, and hence $\theta(\gamma(y))\in B(\olQ)=\phi^{-1}(X_d(\olQ))$. Write $\theta(\gamma(y))=(\bc,\bz)$ in $B(\olQ)=\olQ^{2d-2}$. Fix $\ba\in A_d(\olQ)$ with $[f_{\ba}]=\alpha$, and let $u_1,\ldots,u_{d-1}$ be the finite critical points of $f_{\ba}$ (counted with multiplicity). With the notation of Proposition~\ref{prop:sys}, we have
\[
D_n(\ba)=\sum_{i=1}^{d-1}[f_{\ba}^{\circ n}(u_i)].
\]
Write $\Phi_w(\bc,\bz)=\Omega(\bc,\bz)=(\bb,D)$, where
\[
f_{\bb}(Z)=\lambda^{-1}\left(g_{\bc}(\lambda Z+m_{\bc})-m_{\bc}\right),\quad D=\sum_{j=1}^{d-1}\left[\frac{z_j-m_{\bc}}{\lambda}\right].
\]
Note that
\[
\pi_w(\bb,D)=q(\gamma(y))=q(y)=x=F_d^{\circ n}(s_d(\alpha))=\pi_w(\ba,D_n(\ba)).
\]
Since the fibers of $\pi_w$ on $\olQ$-points are $G_0$-orbits, $(\bb,D)=\zeta\cdot(\ba,D_n(\ba))$ for some $\zeta\in G_0$. Then we deduce that
\[
g_{\bc}=\sigma^{-1}\circ f_{\ba}\circ\sigma\quad\text{and}\quad\sum_{j=1}^{d-1}[z_j]=\sum_{i=1}^{d-1}[\sigma^{-1}\circ f_{\ba}^{\circ n}(u_i)],
\]
where $\sigma(Z)=\zeta(Z-m_{\bc})/\lambda$.

Take a number field $F$ defining $y$, all morphisms involved, and all coordinates above. Let $v\in M_F$. The affine potential of $\overline{H}$ at $v$ is
\[
\psi_v(\bc,\bz)=\max\{\hat{\lambda}_{g_{\bc},v}(c_i),\hat{\lambda}_{g_{\bc},v}(z_j):1\leq i,j\leq d-1\}.
\]
By \eqref{eq:affgreen} and $\hat{\lambda}_{f_{\ba},v}(f_{\ba}^{\circ n}(u_i))=d^n\hat{\lambda}_{f_{\ba},v}(u_i)\geq\hat{\lambda}_{f_{\ba},v}(u_i)$, we obtain
\[
\psi_v(\theta(\gamma(y)))=\psi_v(\bc,\bz)=d^n\max_i\hat{\lambda}_{f_{\ba},v}(u_i).
\]
Summing with the weights $n_v$ implies $h_{\overline{H}}(\theta(\gamma(y)))=d^nh_{\mathrm{crit}}^{\max}(\alpha)$. Thus, we conclude that
\[
h_{\overline{A}}(x)=h_{q^*\overline{A}}(y)=\frac{k_d}{\#\Gamma}\sum_{\gamma\in\Gamma}h_{\overline{H}}(\theta(\gamma(y)))=k_dd^nh_{\mathrm{crit}}^{\max}(\alpha).
\]
If $\alpha$ is PCF, then $h_{\overline{A}}(x)=k_dd^nh_{\mathrm{crit}}^{\max}(\alpha)=0$.

\smallskip

Fix $h\in\Z_{>0}$ such that $A^{\otimes h}$ is very ample, a closed immersion $j:\overline{X}_d\hookrightarrow\P^m_L$, and an isomorphism $j^*\sO(1)\simeq A^{\otimes h}$. Let $S\subseteq M_L$ contain all archimedean places, all places above those $v\in M_\Q$ for which $C_{d,v}\ne0$, and all finite places at which $\Theta_1$ fails to extend to a morphism over the valuation ring. Observe that $S$ is finite. At every place outside $S$, the metric of $\overline{H}_1$ is the canonical model metric, and \eqref{eq:psin} shows that every $\overline{H}_n$, and hence $\overline{H}$, has the same metric. Let $\sX_d$ be the normalization of the scheme-theoretic closure $W$ of $j(\overline{X}_d)$ in $\P^m_{\sO_{L,S}}$, and let $j_S:\sX_d\to\P^m_{\sO_{L,S}}$ be the induced morphism, which is finite onto $W$. Then $\sX_d$ is projective and flat over $\sO_{L,S}$.

After enlarging $S$ if necessary, the line bundle $A$ extends to a line bundle $\sA$ on $\sX_d$, with a fixed identification $\sA|_{\overline{X}_d}\simeq A$. After a further enlargement of $S$ if necessary, we may assume that the fixed isomorphism $j^*\sO(1)\simeq A^{\otimes h}$ and its inverse extend over $\sO_{L,S}$, giving $j_S^*\sO(1)\simeq\sA^{\otimes h}$. Then $\sA$ is relatively ample. The enlarged $S$ is still finite.

Let $\sR$ be the normalization of $\sX_d$ in $\Lambda$, with finite structural morphism $q_S:\sR\to\sX_d$ and generic fiber $R$. For every $\gamma\in\Gamma$, observe that its action on $R$ extends to $\sR$ and satisfies $q_S\circ\gamma=q_S$. After adding finitely many places to $S$ if necessary, the morphism $\theta$ extends to $\theta_S:\sR\to\P^{2d-2}_{\sO_{L,S}}$, and the fixed isomorphism $\theta^*\eta:q^*A\simeq\theta^*\sO(k_d)$ extends to an isomorphism $q_S^*\sA\simeq\theta_S^*\sO(k_d)$.

For every finite place $v\notin S$, the metric of $\overline{H}^{\otimes k_d}$ is the canonical model metric on $\sO(k_d)$. Then the isomorphism $q_S^*\sA\simeq\theta_S^*\sO(k_d)$ identifies the metric of $\overline{N}=\theta^*(\overline{H}^{\otimes k_d})$ with the model metric induced by $q_S^*\sA$. At $v$, every conjugate metric in the definition of $\|\cdot\|_{\mathrm{av},v}$ is the same model metric, and so is $\|\cdot\|_{\mathrm{av},v}$. Since $q^{\mathrm{an}}$ is surjective, we conclude that the $v$-adic metric of $\overline{A}$ is exactly the model metric induced by $(\sX_d,\sA)$.
\end{proof}

\subsection{Relative degree}
We prove the following relative degree formula.
\begin{proposition}\label{prop:deg}
Let $V$ be an integral subvariety of $\MPoly^d$, and let $Y$ be an integral subvariety of $X_d\times_{\MPoly^d}V$ that dominates $V$. Suppose that $G=F_d^{\circ l}|_Y$ maps $Y$ to $Y$ for some $l\geq1$. Set $s=\dim Y-\dim V$. Then $G$ is finite and surjective, with
\[
\deg G=d^{sl}.
\]
For the normalization $\nu:Y^\nu\to Y$, there is a unique finite surjective endomorphism $G^\nu:Y^\nu\to Y^\nu$ satisfying $\nu\circ G^\nu=G\circ\nu$, and we have $\deg G^\nu=d^{sl}$.

Let $D$ be an integral subvariety of $(\MPoly^d)^2$, and let $Z$ be an integral subvariety of $X_d^2\times_{(\MPoly^d)^2}D$ dominating $D$, where the fiber product uses $p_d\times p_d$. Suppose that $H=(F_d^{\circ l}\times F_d^{\circ l})|_Z$ maps $Z$ to $Z$ for some $l\geq1$. Let $\mu:Z^\nu\to Z$ be the normalization, and let $H^\nu:Z^\nu\to Z^\nu$ be the unique lift of $H$. Then
\[
\deg H^\nu=\deg H=d^{(\dim Z-\dim D)l},
\]
and both $H$ and $H^\nu$ are finite and surjective.
\end{proposition}
\begin{proof}
For $a\in\{1,2\}$, write $(B,W,T)=(V,Y,G)$ or $(D,Z,H)$, respectively. Set $t=\dim W-\dim B$. By Proposition~\ref{prop:sys}, $T:W\to W$ is finite, hence the endomorphism $T:W\to W$ is also surjective. Let $q:A_d\times S_{d-1}\to X_d$ be the quotient as in \S~\ref{sec:2.1}. The morphism $(\widetilde{F}_d^{\circ l})^{\times a}$ sends components of $(q^{\times a})^{-1}(W)$ dominating $W$ to such components. Choose a $(\widetilde{F}_d^{\circ l})^{\times a}$-periodic component $W'$ of $(q^{\times a})^{-1}(W)$ dominating $W$, and let $m$ be its exact period. Define
\[
\pi=q^{\times a}|_{W'}\quad\text{and}\quad T'=(\widetilde{F}_d^{\circ lm})^{\times a}|_{W'}.
\]
Then $\pi\circ T'=T^{\circ m}\circ\pi$, and $\pi:W'\to W$ is finite and dominant.

Let $B'\subseteq A_d^a$ be the closure of the image of $W'$ under the parameter projection, and let $\eta=\Spec K$ be its generic point. Since $B'\to B$ is finite and dominant, $\dim W'_\eta=t$. Let $f_1,\ldots,f_a\in K[z]$ be the monic centered polynomials corresponding to $\eta$.

Fix the isomorphism $\Sym^{d-1}(\P^1_K)\simeq\P^{d-1}_K$ determined by
\[
\sum_{i=1}^{d-1}[u_i:v_i]\mapsto[x_0:\cdots:x_{d-1}],\quad\prod_{i=1}^{d-1}(v_i Z_0-u_i Z_1)=\sum_{j=0}^{d-1}(-1)^j x_jZ_0^{d-1-j}Z_1^j.
\]
The chart $\{x_0\ne0\}\subseteq\P^{d-1}_K$ identifies with $S_{d-1,K}$. For $\infty=[1:0]\in\P^1_K$, the hyperplane $H_\infty=\{x_0=0\}\subseteq\P^{d-1}_K$ consists exactly of effective divisors on $\P^1_K$ of degree $d-1$ whose support contains $\infty$.

For $1\leq i\leq a$, define the morphism $\tau_i=\Sym^{d-1}(f_i)$ by
\[
\tau_i:\Sym^{d-1}(\P^1_K)\to\Sym^{d-1}(\P^1_K),\,\sum_{j=1}^{d-1}[z_j]\mapsto\sum_{j=1}^{d-1}[f_i(z_j)].
\]
Let $\chi:(\P^1_K)^{d-1}\to\Sym^{d-1}(\P^1_K)$ be the quotient morphism. For each $1\leq j\leq d-1$, let $\pr_j:(\P^1_K)^{d-1}\to\P^1_K$ be the $j$-th coordinate projection, and define the effective Cartier divisor $E_j=\pr_j^*[\infty]$. Then $\tau_i\circ\chi=\chi\circ f_i^{\times(d-1)}$ for every $1\leq i\leq a$, and $\chi^*H_\infty=\sum_{j=1}^{d-1}E_j$. Observe that
\[
\chi^*(\tau_i^*H_\infty)=\left(f_i^{\times(d-1)}\right)^*\left(\sum_{\nu=1}^{d-1}E_\nu\right)=d\sum_{\nu=1}^{d-1}E_\nu=\chi^*(dH_\infty).
\]
Since $\chi$ is finite and surjective, $\tau_i^*H_\infty=dH_\infty$ and $\tau_i^*\sO(1)\simeq\sO(d)$. Thus, $\tau_i$ is finite.

Define
\begin{align*}
&\sP=\left(\Sym^{d-1}(\P^1_K)\right)^a\simeq(\P^{d-1}_K)^a,\\
&\sT:\sP\to\sP,\,(D_1,\ldots,D_a)\mapsto\left(\tau_1^{\circ lm}(D_1),\ldots,\tau_a^{\circ lm}(D_a)\right),\\
&\sL=\bigotimes_{i=1}^a\pr_i^*\sO(1),
\end{align*}
where $\pr_i:\sP\to\P^{d-1}_K$ are the projections. The morphism $\sT$ is finite, and $\sT^*\sL\simeq\sL^{\otimes d^{lm}}$.

Let $C=\overline{W'_\eta}\subseteq\sP$ have its reduced structure. Since $T'$ is finite and dominant, the base change $T'_\eta$ is surjective. Moreover, $\sT|_{W'_\eta}=T'_\eta$. Note that
\[
\sT(C)=\overline{\sT(W'_\eta)}=\overline{T'_\eta(W'_\eta)}=\overline{W'_\eta}=C.
\]
Then $u=\sT|_C:C\to C$ is finite and surjective. The varieties $W'$ and $C$ have the same function field, and $T'$ and $u$ induce the same field homomorphism. Thus, $\deg u=\deg T'$ and $u_*[C]=(\deg T')[C]$. Since $C$ is integral of dimension $t$ and $\sL|_C$ is ample, $(\sL^t\cdot C)>0$. The projection formula shows
\[
\deg T'=\deg u=\frac{((\sT^*\sL)^t\cdot C)}{(\sL^t\cdot C)}=\frac{d^{tlm}(\sL^t\cdot C)}{(\sL^t\cdot C)}=d^{tlm}.
\]
From $\pi\circ T'=T^{\circ m}\circ\pi$, we obtain $\deg T'=(\deg T)^m$. Therefore, $\deg T=d^{tl}$ in both cases.

Write $W=\Spec R$, and let $R^\nu$ be its integral closure in $\operatorname{Frac}(R)$. Integrality implies $T^*(R^\nu)\subseteq R^\nu$. Note that $R^\nu$ is finite over $T^*(R^\nu)$. Thus, the induced endomorphism of $\Spec R^\nu$ is finite and dominant. Its degree equals $\deg T$ because normalization preserves function fields.
\end{proof}

\section{Intertwined polynomials and small-point correspondences}\label{sec:intcorr}
In \S~\ref{sec:int}, we define and study the intertwined relation. In \S~\ref{sec:tube}, we adapt the correspondence method of \cite[\S~6]{JSX26} to Galois-invariant sets of small points contained in a closed diagonal or Frobenius tube. In \S~\ref{sec:local}, we establish the local estimates in the wild and tame cases.
\subsection{Intertwined polynomials}\label{sec:int}
Fix an integer $d\geq2$. A polynomial $f\in\C[z]$ of degree $d$ is exceptional in the sense of Definition~\ref{def:excep} if and only if it is affinely conjugate to $z^d$, $T_d$, or $-T_d$. The exceptional locus in $\MPoly^d(\C)$ is 
\[
E_d=\left\lbrace[z^d],[T_d],[-T_d]\right\rbrace\subseteq\MPoly^d(\C).
\]
For even $d$, we have $[T_d]=[-T_d]$ and $\# E_d=2$. For odd $d$, $\# E_d=3$.

Two polynomials $f,g\in\C[z]$ of degree $d$ are called \emph{intertwined}, written $f\approx g$, if there is a reduced algebraic curve $Z\subseteq\A^2_\C$ such that $(f,g)(Z)\subseteq Z$ and each irreducible component of $Z$ projects surjectively onto both axes. Here
\[
(f,g):\A^2_\C\to\A^2_\C,\,(x,y)\mapsto(f(x),g(y)).
\]
By \cite[Theorem~3.39]{FG22}, this condition is equivalent to the existence of $N\in\Z_{>0}$ and $u,v,h\in\C[z]\setminus\C$ such that
\begin{equation}\label{eq:semi}
f^{\circ N}\circ u=u\circ h\quad\text{and}\quad g^{\circ N}\circ v=v\circ h.
\end{equation}
The same theorem shows that the intertwined relation is invariant under affine conjugacy, defines an equivalence relation on $\MPoly^d_\C$, and preserves exceptionality.

We will need the following proposition bounding the cardinalities of intertwined classes, which follows directly from \cite[Theorem~3.46]{FG22}:
\begin{proposition}\label{prop:Md}
Let $d\geq2$ be an integer. Then there exists an integer $M_d\geq1$, depending only on $d$, such that every intertwined class in $\MPoly^d(\C)$ has at most $M_d$ elements.
\end{proposition}

In the following proposition, we show that the intertwined relation is constructible and prove a property relating it to rational functions. 
\begin{proposition}\label{prop:rel}
\textup{(1)} The intertwined relation $\approx$ is a constructible subset of $(\MPoly^d)^2$, defined over $\Q$.

\textup{(2)} Let $V\subseteq\MPoly^d$ be a positive-dimensional geometrically integral subvariety, defined over a number field $L$. Then there exists $u_V\in L(V)\setminus L$ such that, for every irreducible component $Z$ of the closure of the intertwined relation in $V\times_L V$ whose two projections $p_1,p_2:Z\to V$ are dominant, one has
\[
p_1^*(u_V)=p_2^*(u_V)\quad\text{in }L(Z).
\]
\end{proposition}
\begin{proof}
\textup{(1)} By \cite[Theorems~3.43 and 3.44]{FG22}, there are integers $N_d,J_d\geq1$, depending only on $d$, such that \eqref{eq:semi} holds with
\[
N\leq N_d\quad\text{and}\quad\deg u,\deg v\leq J_d
\]
whenever $f\approx g$. For $1\leq N\leq N_d$ and $1\leq e_1,e_2\leq J_d$, let
\[
I_{N,e_1,e_2}\subseteq (\Poly^d)^2\times\Poly^{e_1}\times\Poly^{e_2}\times\Poly^{d^N}
\]
be the locus of $(f,g,u,v,h)$ defined by \eqref{eq:semi}, where $(\deg u,\deg v,\deg h)=(e_1,e_2,d^N)$. Clearly, $I_{N,e_1,e_2}$ is a closed subset defined over $\Q$. Let $\Psi:\Poly^d\to\MPoly^d$ be the quotient. Then the intertwined relation in $(\MPoly^d)^2$ is
\[
\sI=\bigcup_{\substack{1\leq N\leq N_d\\1\leq e_1,e_2\leq J_d}}(\Psi\times\Psi)\left(\pr_{12}(I_{N,e_1,e_2})\right),
\]
where $\pr_{12}:(\Poly^d)^2\times\Poly^{e_1}\times\Poly^{e_2}\times\Poly^{d^N}\to (\Poly^d)^2$ is the projection to the first two factors. Then the conclusion of \textup{(1)} follows from Chevalley's theorem.

\textup{(2)} Let $M_d$ be as in Proposition~\ref{prop:Md}. Set $F=L(V)$. Since $\operatorname{trdeg}_L F=\dim V>0$, choose $x\in F$ transcendental over $L$ and a nonempty affine $L$-open $U\subseteq V$ on which $x$ is regular. Let $\eta\in U(F)$ be the generic point. Fix an algebraic closure $\overline{F}$ of $F$ inside $\C$. The set
\[
S_\eta=\{z\in U(\overline{F}):z\approx\eta\}
\]
contains $\eta$, has at most $M_d$ elements, and is $\Gal(\overline{F}/F)$-invariant. Therefore,
\[
P_\eta(T):=\prod_{z\in S_\eta}(T-x(z))\in F[T],\quad P_\eta(x)=0.
\]
If every coefficient belongs to $L$, this monic polynomial makes $x$ algebraic over $L$, a contradiction. Choose a coefficient $u_V\in F\setminus L$.

Let $Z$ be as in the statement. Set $E=L(Z)$ and fix an algebraic closure $\overline{E}$ of $E$ inside $\C$. The embeddings $p_i^*:F\hookrightarrow E\hookrightarrow\overline{E}$ define points $\eta_i\in U(\overline{E})$. By \textup{(1)}, a nonempty open subset of $Z$ is contained in the intertwined relation, so $\eta_1\approx\eta_2$. Base change of the finite generic fibers implies
\[
p_1^*P_\eta(T)=\prod_{\substack{z\in U(\overline{E})\\z\approx\eta_1}}(T-x(z))=\prod_{\substack{z\in U(\overline{E})\\z\approx\eta_2}}(T-x(z))=p_2^*P_\eta(T),
\]
where $p_i^*$ acts coefficientwise on $P_\eta(T)$. Both polynomials belong to $E[T]$. Comparing the coefficient chosen as $u_V$ therefore gives $p_1^*(u_V)=p_2^*(u_V)$ in $L(Z)$, as required.
\end{proof}

\subsection{Closed tubes}\label{sec:tube}
Let $L$ be a number field and $v$ a finite place of $L$ with $L_v=\Q_p$. Let $X$ be a normal, projective, and geometrically integral variety defined over $L$, with $\dim X=n\geq1$. Let $\overline{A}$ be a nef adelic line bundle on $X$, with $A$ ample. Suppose its metric at $v$ is induced by a relatively ample line bundle $\sA$ extending $A_{L_v}$ on a normal projective flat model $\sX/\Z_p$. Suppose that $X_0:=\sX_{\F_p}$ is geometrically integral.

Choose $b\in\Z_{>0}$ such that $\sA^{\otimes b}$ is relatively very ample, and sections $s_0,\ldots,s_N\in H^0(\sX,\sA^{\otimes b})$ defining a closed immersion $j:\sX\hookrightarrow\P^N_{\Z_p}$ with $j^*\sO(1)\simeq\sA^{\otimes b}$. Let $[x_0:\cdots:x_N]$ and $[y_0:\cdots:y_N]$ denote the coordinates on the first and second factors of $\P^N_{L_v}\times_{L_v}\P^N_{L_v}$.

For $0\leq i\leq N$, let $E_i$ be the affinoid domain of $(X_{L_v}\times_{L_v}X_{L_v})^{\mathrm{an}}$ where
\[
x_i y_i\ne0\quad\text{and}\quad\left|x_j/x_i\right|,\left|y_j/y_i\right|\leq1\quad(0\leq j\leq N).
\]
We use the Berkovich analytification. For $e\in\{1,p\}$ and $0<\delta<1$, define
\[
T_e(\delta)=\bigcup_{i=0}^N\left\{z\in E_i:\max_{0\leq j\leq N}\left|\frac{y_j}{y_i}-\left(\frac{x_j}{x_i}\right)^e\right|(z)\leq\delta\right\}.
\]
Then $T_e(\delta)$ is closed. Let $F_1=\id_{X_0}$, let $F_p$ be the absolute Frobenius of $X_0$, and let
\[
\Gamma_e\subseteq X_0\times_{\F_p}X_0
\]
be the graph of $F_e$. Then $\red(T_e(\delta))\subseteq|\Gamma_e|$, where $\red$ is the reduction for the proper model $\sX\times_{\Z_p}\sX$.

The following proposition extends the correspondence argument of \cite[Theorem~6.6 and the proof of Lemma~6.14]{JSX26} to Galois-invariant sets of small points.
\begin{proposition}\label{prop:tube}
With the notation and hypotheses above, fix $e\in\{1,p\}$ and $0<\delta<1$. Let
\[
B\subseteq(X\times_LX)(\olQ)\cap T_e(\delta)
\]
be $\Gal(\olQ/L)$-invariant. Define
\[
B_\epsilon=\left\lbrace(x,y)\in B\colon h_{\overline{A}}(x)+h_{\overline{A}}(y)<\epsilon\right\rbrace\quad(\epsilon>0)\quad\text{and}\quad Z(B)=\bigcap_{n=1}^\infty\overline{B_{1/n}},
\]
where the closures are taken in $(X\times_LX)_{\olQ}$. Suppose that $C$ is an irreducible component of $Z(B)$ such that the two projections $p_1,p_2:C\to X_{\olQ}$ are dominant. Let $\nu:W\to C$ be the normalization, and set $\pi_i=p_i\circ\nu:W\to X_{\olQ}$ for $1\leq i\leq 2$.

Then $p_1,p_2,\pi_1,\pi_2$ are finite and surjective, and
\begin{equation}\label{eq:num}
\pi_2^*A_{\olQ}\equiv e\pi_1^*A_{\olQ}.
\end{equation}
\end{proposition}
\begin{proof}
By Noetherianity, there is a positive integer $n_0\geq1$ such that $Z(B)=\overline{B_{1/n}}$ for all $n\geq n_0$. Then 
\begin{equation}\label{eq:BepsC}
\overline{B_\epsilon\cap C}=C
\end{equation}
for all $\epsilon>0$.

Take a finite extension $L'/L$ over which $C$ is defined, and a place $w$ above $v$. Base-change the global data to $L'$, and write $(L,v)$ for $(L',w)$ in the remainder of the proof. Set
\[
R=\sO_{L_v},\quad\kappa=R/\mathfrak m_R,\quad\text{and}\quad m=\dim W=\dim C\geq1.
\]
Denote the base change of the original graph $\Gamma_e$ over $\F_p$ to $\kappa$ by $\Gamma_{e,\kappa}=\Gamma_e\times_{\F_p}\kappa$. For $e=p$, the graph of the morphism
\[
F_{p,\kappa}=F_{X_0}\times_{\F_p}\id_\kappa:X_0\times_{\F_p}\kappa\to X_0\times_{\F_p}\kappa
\]
is $\Gamma_{p,\kappa}$. The line bundle $\sA_\kappa$ comes from $\F_p$, so $F_{e,\kappa}^*\sA_\kappa\simeq\sA_\kappa^{\otimes e}$.

Define
\[
\overline{H}=\pi_1^*\overline{A}\otimes\pi_2^*\overline{A},
\]
which is an adelic line bundle on $W/L$. Note that $\overline{H}$ is nef and $H$ is ample. With the help of \eqref{eq:BepsC}, since $\olQ$ is countable and $m\geq1$, it is clear that there is a generic sequence
\[
(z_j)_{j=1}^\infty\subseteq\nu^{-1}(B\cap C)
\]
in $W/\olQ$ such that
\[
0\leq h_{\overline{H}}(z_j)<1/j\quad(j\geq1).
\]
Then the fundamental inequality \cite[Theorem~1.10]{Zhang1995} shows
\[
0\leq h_{\overline{H}}(W)\leq\liminf_{j\to\infty}h_{\overline{H}}(z_j)=0.
\]

Let $\sW$ be the normal projective flat $R$-model obtained by normalizing the scheme-theoretic closure of $C_{L_v}$ in $\sX_R\times_R\sX_R$. Let
\[
\Phi=(\Pi_1,\Pi_2):\sW\to\sX_R\times_R\sX_R
\]
be the corresponding finite morphism, which extends $\phi=(\pi_1,\pi_2)_{L_v}$. The line bundle
\[
\sH=\Pi_1^*\sA_R\otimes\Pi_2^*\sA_R
\]
is relatively ample and induces the metric of $\overline{H}$ at $v$. Write
\[
[\sW_\kappa]=\sum_{k=1}^s a_k[Y_k],
\]
where $Y_k$ are the integral components and $a_k\in\Z_{>0}$ their multiplicities. For $1\leq k\leq s$, let $\eta_k$ be the generic point of $Y_k$ and $\xi_k\in W_{L_v}^{\mathrm{an}}$ the unique divisorial analytic point satisfying $\red_{\sW}(\xi_k)=\eta_k$; see \cite[\S~2.3]{CL2006}. The model measure formula \cite[D\'efinition~2.4]{CL2006} shows
\[
\mu_v
:=
\frac{c_1(\overline{H})_v^m}{(H^m\cdot W)}
=
\frac{1}{(H^m\cdot W)}
\sum_{k=1}^s
a_k(\sH^m\cdot Y_k)\delta_{\xi_k},
\]
where $\delta_z$ denotes the Dirac measure at $z$. Note that $a_k(\sH^m\cdot Y_k)>0$ for $1\leq k\leq s$ because $\sH$ is relatively ample. For $j\geq1$, set $O_j=\Gal(\olQ/L)\cdot z_j$, and define the probability measure
\[
\mu_j=\frac{1}{\#O_j}\sum_{z\in O_j}\delta_z
\]
on $W_{L_v}^{\mathrm{an}}$. Yuan's equidistribution theorem \cite[Theorem~3.1]{Yuan2008} implies that
\[
\mu_j\to\mu_v\quad(j\to\infty)
\]
weakly on $W_{L_v}^{\mathrm{an}}$. Since $B$ is Galois-invariant, for $j\geq1$, the support $\Supp(\mu_j)$ is contained in the closed subset $(\phi^{\mathrm{an}})^{-1}(T_e(\delta))$. Thus, we have $\Supp(\mu_v)\subseteq(\phi^{\mathrm{an}})^{-1}(T_e(\delta))$, and then
\begin{equation}\label{eq:xitube}
\phi^{\mathrm{an}}(\xi_k)\in T_e(\delta)\quad(1\leq k\leq s).
\end{equation}

For $1\leq k\leq s$, we have
\[
\Phi_\kappa(\eta_k)=\Phi_\kappa(\red_{\sW}(\xi_k))=\red_{\sX\times\sX}\left(\phi^{\mathrm{an}}(\xi_k)\right),
\]
where $\Phi_\kappa:\sW_\kappa\to\sX_\kappa\times_\kappa\sX_\kappa$. By the definition of $T_e(\delta)$, we obtain $\red_{\sX\times\sX}(T_e(\delta))\subseteq\Gamma_{e,\kappa}$ as sets. Thus, \eqref{eq:xitube} yields
\begin{equation}\label{eq:etak}
\Phi_\kappa(\eta_k)\in\Gamma_{e,\kappa}\quad(1\leq k\leq s).
\end{equation}
Set $W_0=(\sW_\kappa)_{\red}=\bigcup_{k=1}^sY_k$. For $1\leq k\leq s$, from \eqref{eq:etak} we deduce that
\[
(\Phi_\kappa|_{Y_k})^{-1}(\Gamma_{e,\kappa})=Y_k
\]
as sets. Thus, $\Phi_\kappa(W_0)\subseteq\Gamma_{e,\kappa}$ as sets. Since $W_0$ is reduced, we conclude that there is a morphism $\varphi:W_0\to\Gamma_{e,\kappa}$ such that
\[
\Phi_\kappa|_{W_0}=\iota_{e,\kappa}\circ\varphi:W_0\to\sX_\kappa\times_\kappa\sX_\kappa,
\]
where $\iota_{e,\kappa}:\Gamma_{e,\kappa}\to\sX_\kappa\times_\kappa\sX_\kappa$ is the inclusion. Since $\Phi_\kappa|_{W_0}$ is finite and $\iota_{e,\kappa}$ is a closed immersion, the morphism $\varphi:W_0\to\Gamma_{e,\kappa}$ is finite. Both projections $\Gamma_{e,\kappa}\to\sX_\kappa$ are finite, so the two projections $W_0\to\sX_\kappa$ are finite. Let $1\leq i\leq2$. Then every fiber of $\Pi_{i,\kappa}:\sW_\kappa\to\sX_\kappa$ is zero-dimensional. Upper semicontinuity gives a closed subset
\[
Z_i=\{x\in\sX_R\colon \dim\Pi_i^{-1}(x)\geq1\}
\]
of $\sX_R$ such that $Z_i\cap\sX_\kappa=\emptyset$. Since $\sX_R$ is proper over $R$, we conclude that $Z_i=\emptyset$. So $\Pi_i$ is quasi-finite. Since $\Pi_i$ is proper, it is finite. Therefore, its generic fiber $\pi_i:W\to X$ is finite. Then $\pi_i$ is surjective because it is dominant. Clearly, $p_i$ is also finite and surjective.

\smallskip

Let $\rho_1,\rho_2:\Gamma_{e,\kappa}\to\sX_\kappa$ be the two projections. Then
\[
\rho_2=F_{e,\kappa}\circ\rho_1\quad\text{and}\quad\Pi_{i,\kappa}|_{W_0}=\rho_i\circ\varphi\quad(1\leq i\leq 2).
\]
Since $F_{e,\kappa}^*\sA_\kappa\simeq\sA_\kappa^{\otimes e}$, the line bundle
\[
\sN:=\Pi_2^*\sA_R\otimes(\Pi_1^*\sA_R)^{\otimes(-e)}
\]
satisfies
\[
\sN|_{W_0}\simeq\varphi^*\rho_1^*\left(F_{e,\kappa}^*\sA_\kappa\otimes\sA_\kappa^{\otimes(-e)}\right)\simeq\sO_{W_0}.
\]
Let $D\subseteq W_{L_v}$ be an integral curve, and $\sD\subseteq\sW$ its scheme-theoretic closure. Then $\sD$ is projective and flat over $R$. Write the fundamental one-cycle associated with $\sD_\kappa$ as
\[
[\sD_\kappa]_1=\sum_{l=1}^t b_l[D_l],\quad b_l=\operatorname{length}\sO_{\sD_\kappa,\varrho_l}>0,
\]
where $D_l$ ($1\leq l\leq t$) are the irreducible components of $(\sD_\kappa)_{\red}$ with generic point $\varrho_l$, respectively. Since each $D_l\subseteq W_0$, by the specialization of degree \cite[\S~20.3]{Fulton1998} we obtain
\[
\deg(\sN|_D)=\sum_{l=1}^t b_l\deg(\sN|_{D_l})=\sum_{l=1}^t b_l\cdot0=0.
\]
Every integral curve in $W_{\olQ}$ is defined over a finite extension $F/L$; after taking $w\in M_F$ over $v$, the same argument for the base-changed model over $\sO_{F_w}$ shows that $\sN_{\sO_{F_w}}$ has degree zero on that curve. By the definition of $\sN$, we deduce that $\pi_2^*A_{\olQ}\equiv e\pi_1^*A_{\olQ}$.
\end{proof}

\begin{lemma}\label{lem:inv}
Let $X$ be a normal geometrically integral projective variety defined over a number field $L$, and let $H$ be an ample line bundle on $X$. If $\pi_1,\pi_2:W\to X$ are two finite surjective morphisms from a normal geometrically integral projective variety $W$ such that $\pi_2^*H\equiv q\pi_1^*H$ for some real number $q>1$, then there is no $U\in L(X)\setminus L$ satisfying $\pi_1^*(U)=\pi_2^*(U)$.
\end{lemma}
\begin{proof}
Suppose that $U\in L(X)\setminus L$ satisfies $\pi_1^*(U)=\pi_2^*(U)$. Set $n=\dim X$. Let $D$ be the pole divisor of $U$, regarded as an effective Weil divisor. Since $X$ is normal and projective and $U$ is nonconstant, $D\neq0$. From $\pi_1^*(U)=\pi_2^*(U)$, we obtain $\pi_1^*D=\pi_2^*D$. Set $e_i=\deg\pi_i$ ($1\leq i\leq 2$) and $a=(D\cdot H^{n-1})>0$. By the projection formula,
\[
e_2a=(\pi_2^*D\cdot(\pi_2^*H)^{n-1})=q^{n-1}(\pi_1^*D\cdot(\pi_1^*H)^{n-1})=q^{n-1}e_1a.
\]
So $e_2=q^{n-1}e_1$. On the other hand, top intersections show $e_2=q^ne_1>q^{n-1}e_1$, a contradiction.
\end{proof}

\subsection{Local estimates}\label{sec:local}
We establish the necessary local estimates in the wild and tame cases.
\begin{lemma}\label{lem:in}
Let $E/\Q_p$ be a finite abelian extension, where $p\geq3$ is a rational prime.

\textup{(1)} If $E/\Q_p$ is wildly ramified, then the inertia group $I(E/\Q_p)$ contains a subgroup $G$ of order $p$ such that for every $\tau\in G$ and every $x\in\sO_E$, one has
\begin{equation}\label{eq:in}
|\tau(x)-x|_p\leq p^{-1/(p-1)}.
\end{equation}

\textup{(2)} If $E/\Q_p$ is tamely ramified, then the ramification index of $E/\Q_p$ divides $p-1$.
\end{lemma}
\begin{proof}
\textup{(1)} Suppose that $E/\Q_p$ is wildly ramified. For $n\in\Z_{>0}$, let $\zeta_n$ be a primitive $n$-th root of unity. We may choose $\zeta_{p^j}^p=\zeta_{p^{j-1}}$ for all $j\geq1$. By local Kronecker--Weber \cite[Chapter~\Rom{5}, Corollary~1.9]{Neukirch1999}, it is easy to see that there exist a finite unramified extension $F/\Q_p$ and an integer $m\geq0$ such that $E\subseteq F_m:=F(\zeta_{p^m})$. Fix such a finite unramified extension $F/\Q_p$. Let $m\geq0$ be the minimal integer such that $E\subseteq F_m$. Since $E/\Q_p$ is wildly ramified and $F/\Q_p$ is unramified, we see that $p$ divides the ramification index $e(F_m/F)$, which implies $m\geq2$. Note that
\[
H:=\Gal(F_m/F_{m-1})=\{\sigma_j\colon0\leq j<p\}\cong\Z/p\Z,
\]
where $\sigma_j|_F=\id_F$ and $\sigma_j(\zeta_{p^m})=\zeta_p^j\zeta_{p^m}$. Since $E/\Q_p$ is Galois, by restriction we obtain a map
\[
\iota:H\to\Gal(E/\Q_p).
\]
If $\iota(H)=1$, then $E\subseteq F_m^H=F_{m-1}$, contradicting the minimality of $m$. Thus, we conclude that $\iota$ is injective. Then the subgroup $G:=\iota(H)$ of $\Gal(E/\Q_p)$ has order $p$. Since $F_m/F$ is totally ramified and every $\sigma_j$ fixes $F$, we deduce that $G\subseteq I(E/\Q_p)$.

Since $F/\Q_p$ is unramified, $p\sO_F$ is the maximal ideal of $\sO_F$. Observe that the polynomial $\Phi_{p^m}(T+1)$ is Eisenstein at $p\sO_F$, where $\Phi_{p^m}$ is the $p^m$-th cyclotomic polynomial. By \cite[\S~I.6, Proposition~17 and its Corollary]{Serre1979}, we conclude that
\[
\sO_{F_m}=\sO_F[\zeta_{p^m}-1]=\sO_F[\zeta_{p^m}].
\]
For every $\sigma\in H$, we have
\[
(\sigma-1)\sO_{F_m}\subseteq(\zeta_p-1)\sO_{F_m}.
\]
Since $|\zeta_p-1|_p=p^{-1/(p-1)}$ and $\sO_E\subseteq\sO_{F_m}$, we deduce \eqref{eq:in} for every $\tau\in G$.

\smallskip

\textup{(2)} The conclusion follows from \cite[Proposition~2.2.9]{Newton2012}.
\end{proof}

\section{Proof of the main theorem}\label{sec:proof}
In this section, we complete the proof of the main theorem.
\subsection{Orbit closures and fields of definition}
The following lemma formulates the orbit-closure argument needed.
\begin{lemma}\label{lem:orb}
Let $M$ be an integral variety over an algebraically closed field $k$. Let $p:X\to M$ be a morphism of $k$-varieties with a section $s:M\to X$, and $F:X\to X$ a finite morphism with $p\circ F=p$. For $n\geq0$, define $\alpha_n=F^{\circ n}\circ s$ and $C_n=\alpha_n(M)$; let
\[
Z_n=\overline{\bigcup_{m\geq n}C_m}
\]
have the reduced structure. Then there is $N_0\geq0$ such that $Z_N=Z_{N_0}$ for every $N\geq N_0$. Set $Z=Z_{N_0}$. For every irreducible component $T$ of $Z$, one has
\[
p(T)=M\quad\text{and}\quad T=\overline{\bigcup_{\substack{n\geq N_0\\C_n\subseteq T}}C_n}.
\]
Moreover, either $\alpha_a=\alpha_b$ for some integers $b>a\geq0$, or there are an irreducible component $Y_0$ of $Z$ and an integer $l\geq1$ satisfying
\[
p(Y_0)=M,\quad\dim Y_0>\dim M,\quad\text{and}\quad F^{\circ l}(Y_0)=Y_0.
\]
\end{lemma}
\begin{proof}
For $n\geq0$, since $p\circ\alpha_n=\id_M$, every $\alpha_n$ is a section of the separated morphism $p$, hence a closed immersion \cite[Tag~01KT]{Stacks}. Thus, $C_n$ is a closed subvariety of $X$ isomorphic to $M$. The descending chain of closed subsets
\[
Z_0\supseteq Z_1\supseteq Z_2\supseteq\cdots
\]
stabilizes, so there is $N_0\geq0$ such that $Z_N=Z_{N_0}$ for every $N\geq N_0$. Set $Z=Z_{N_0}$. Since $F$ is finite, hence closed, we see that $F(Z)=Z$.

Let $\Irr(Z)$ be the set of irreducible components of $Z$ and $T\in\Irr(Z)$. Define
\[
T^\circ=T\setminus\bigcup_{\substack{T'\in\Irr(Z)\\T'\neq T}}T'=Z\setminus\bigcup_{\substack{T'\in\Irr(Z)\\T'\neq T}}T',
\]
which is nonempty and Zariski open in $Z$. Take $n_0\geq N_0$ such that $C_{n_0}\cap T^\circ\neq\emptyset$. Since $C_{n_0}$ is irreducible and $C_{n_0}\subseteq Z$, we obtain $C_{n_0}\subseteq T$. Thus,
\[
M=p(C_{n_0})\subseteq p(T)\subseteq M,
\]
so $p(T)=M$. Applying the same argument to every nonempty Zariski open subset of $T^\circ$ implies that
\[
T=\overline{\bigcup_{\substack{n\geq N_0\\C_n\subseteq T}}C_n}.
\]

Suppose that $\alpha_a\neq\alpha_b$ for all integers $b>a\geq0$. Since every $C_n$ is irreducible of dimension $\dim M$ with $p(C_n)=M$, we deduce that $p(T)=M$ and $\dim T\geq\dim M$ for all $T\in\Irr(Z)$. If every $T\in\Irr(Z)$ has dimension $\dim M$, then for each $n\geq N_0$, we see that $C_n\in\Irr(Z)$. Thus, for some integers $b>a\geq N_0$, we have $C_a=C_b=:C$. Since $p\circ\alpha_a=p\circ\alpha_b=\id_M$, both
\[
\alpha_a:M\to C\quad\text{and}\quad\alpha_b:M\to C
\]
are inverse to the isomorphism $p|_C:C\to M$. So $\alpha_a=\alpha_b$, a contradiction. Consequently,
\[
r:=\max_{T\in\Irr(Z)}\dim T>\dim M.
\]
Fix $T_0\in\Irr(Z)$ with $\dim T_0=r$. Since $F$ is finite and $F(Z)=Z$, maximality of $r$ shows
\[
\{F^{\circ n}(T_0)\colon n\geq0\}\subseteq\Irr(Z).
\]
Thus, there are integers $m\geq0$ and $l\geq1$ such that $Y_0:=F^{\circ m}(T_0)\in\Irr(Z)$ and $F^{\circ l}(Y_0)=Y_0$.
\end{proof}

In the following proposition, we apply Lemma~\ref{lem:orb} to the setting of the main theorem and make several reductions. A key consequence is the identity $L(x)=L(\varphi(x))$ in \eqref{eq:q}.
\begin{proposition}\label{prop:orb}
Let $d\geq2$ be an integer, and let $K$ be a number field. Suppose that there is a sequence $(x_j)_{j=1}^\infty\subseteq\MPoly^d(\Kab)$ of pairwise distinct points with
\[
\lim_{j\to\infty}h_{\mathrm{crit}}(x_j)=0.
\]
Then there are
\begin{itemize}
\item a finite extension $L/K$,
\item a positive-dimensional geometrically integral subvariety $V\subseteq\MPoly^d_L$ and a generic sequence $(a_j)_{j=1}^\infty\subseteq V(\Lab)$ in $V/\olQ$ of pairwise distinct points, with $\lim_{j\to\infty}h_{\mathrm{crit}}(a_j)=0$,
\item a normal geometrically integral variety $Y/L$, a surjective morphism $\varphi:Y\to V$, and a finite endomorphism $G:Y\to Y$,
\item a normal projective compactification $Y\subseteq\overline{Y}$ and a nef adelic line bundle $\overline{A}_Y$ on $\overline{Y}$, with $A_Y$ ample and the metric of $\overline{A}_Y$ induced by a relatively ample model at all but finitely many places of $L$,
\end{itemize}
satisfying the following two properties.

\textup{(i)} $\varphi\circ G=\varphi$, $s:=\dim Y-\dim V>0$, and $\deg G=d^{sl}$ for some $l\in\Z_{>0}$.

\textup{(ii)} For every $\Gal(\Lab/L)$-invariant subset $P\subseteq V(\Lab)$ satisfying
\begin{equation}\label{eq:pdense}
\overline{\{a\in P\colon h_{\mathrm{crit}}(a)<\epsilon\}}=V_{\olQ},\quad\forall\epsilon>0,
\end{equation}
there is a $\Gal(\Lab/L)$-invariant subset $Q_P\subseteq Y(\Lab)\cap\varphi^{-1}(P)$ that is Zariski dense in $Y_{\olQ}$ such that
\begin{equation}\label{eq:qdense}
\overline{\{x\in Q_P\colon h_{\overline{A}_Y}(x)<\epsilon\}}=Y_{\olQ},\quad\forall\epsilon>0,
\end{equation}
and 
\begin{equation}\label{eq:q}
G(Q_P)\subseteq Q_P,\quad L(x)=L(\varphi(x)),\quad h_{\overline{A}_Y}(G(x))=d^l h_{\overline{A}_Y}(x)\quad(x\in Q_P).
\end{equation}
\end{proposition}
\begin{proof}
Take an irreducible subvariety $V_{\olQ}\subseteq\MPoly^d_{\olQ}$ of smallest dimension containing infinitely many terms of the given sequence $(x_j)_{j=1}^\infty$. Then $\dim V_{\olQ}\geq1$. Denote the subsequence of $(x_j)_{j=1}^\infty$ consisting of points in $V_{\olQ}$ by $(a_j)_{j=1}^\infty$. By the minimality of $\dim V_{\olQ}$, we see that $(a_j)_{j=1}^\infty$ is a generic sequence in $V_{\olQ}/\olQ$, consisting of pairwise distinct points.

Fix a finite extension $L/K$ such that
\begin{itemize}
\item there is a geometrically integral model $V$ of $V_{\olQ}$ over $L$, and we fix it;
\item $\Q(\lambda,\zeta_{d-1})\subseteq L$, where $\lambda\in\R_{>0}\cap\olQ$ is such that $\lambda^{d-1}=d$, and $\zeta_{d-1}$ is a fixed primitive $(d-1)$-th root of unity.
\end{itemize}
Then $(a_j)_{j=1}^\infty\subseteq V(\Lab)$ is a generic sequence in $V/\olQ$ because $\Kab\subseteq\Lab$.

We use the constructions of Proposition~\ref{prop:sys} and Theorem~\ref{thm:met} over $L$. The proof of Theorem~\ref{thm:met} shows that the fixed number field $L$ works for Theorem~\ref{thm:met}.

Set
\begin{align*}
X_V&=\left(X_d\times_{\MPoly^d}V\right)_{\mathrm{red}},\quad &&F_V=F_d|_{X_V}:X_V\to X_V,\\
p_V&=p_d|_{X_V}:X_V\to V,\quad &&s_V=s_d|_V:V\to X_V.
\end{align*}

For $n\geq0$, define
\[
\alpha_n=F_V^{\circ n}\circ s_V\quad\text{and}\quad C_n=\alpha_n(V).
\]
Let $N_0\geq0$ and $Z/\olQ$ be as in Lemma~\ref{lem:orb}. If
\[
\alpha_a=F_V^{\circ a}\circ s_V=F_V^{\circ b}\circ s_V=\alpha_b
\]
for some integers $b>a\geq0$, then \eqref{eq:pcf} implies that every point of $V(\C)$ is PCF, contradicting Thurston's rigidity (see \cite[Theorem~2.2]{BIJL2014}). Thus, by Lemma~\ref{lem:orb}, there are an irreducible component $Y_0$ of $Z$ (over $\olQ$) and an integer $l\geq1$ such that
\[
p_V(Y_0)=V,\quad \dim Y_0>\dim V,\quad F_V^{\circ l}(Y_0)=Y_0.
\]
After replacing $L$ with a finite extension, we may assume that $Y_0$ is defined over $L$.

Let $\nu:Y\to Y_0$ be the normalization. Then $Y$ is normal and geometrically integral. Set
\[
s:=\dim Y-\dim V=\dim Y_0-\dim V>0.
\]
Define
\[
\varphi=p_V|_{Y_0}\circ\nu:Y\to V\quad\text{and}\quad f_0=F_V^{\circ l}|_{Y_0}:Y_0\to Y_0.
\]
Then $\varphi(Y)=V$. Let $G:Y\to Y$ be the unique lift satisfying $\nu\circ G=f_0\circ\nu$. Then $\varphi\circ G=\varphi$. By Proposition~\ref{prop:deg}, both $f_0$ and $G$ are finite and surjective, and
\[
\deg G=d^{sl}.
\]
Thus, \textup{(i)} holds.

Let $\overline{Y}$ be the normalization of the closure of $Y_0$ in $\overline{X}_{d,L}$, and let $\overline{\nu}:\overline{Y}\to\overline{X}_{d,L}$ be the induced finite morphism. Note that $Y,\nu,G,\overline{Y},\overline{\nu}$ are still defined over $L$, and that $\overline{Y}$ is a normal projective compactification of $Y$. Then
\[
\overline{A}_Y:=\overline{\nu}^*\overline{A}
\]
is a nef adelic line bundle on $\overline{Y}/L$, and its underlying line bundle $A_Y=\overline{\nu}^*A$ is ample. Outside a finite set of places of $L$, $\overline{\nu}$ extends to a finite morphism of projective models in Theorem~\ref{thm:met}, and the pullback of the relatively ample model $\sA$ of $A$ is relatively ample. Thus, by Theorem~\ref{thm:met}, the metric of $\overline{A}_Y$ is induced by a relatively ample model at all but finitely many places of $L$.

Define $I=\{n\geq N_0\colon C_n\subseteq Y_0\}$. By Lemma~\ref{lem:orb}, we obtain
\begin{equation}\label{eq:Cnden}
Y_0=\overline{\bigcup_{n\in I}C_n}.
\end{equation}
Take a nonempty Zariski open subset $U\subseteq Y_0$ defined over $L$ such that
\[
\nu_U:=\nu|_{\nu^{-1}(U)}:\nu^{-1}(U)\to U
\]
is an isomorphism over $L$.

We consider \textup{(ii)}. Let $P\subseteq V(\Lab)$ be $\Gal(\Lab/L)$-invariant and satisfy \eqref{eq:pdense}. Define
\[
Q_P^0=\{\nu_U^{-1}(\alpha_n(a))\colon n\in I,\ a\in P,\ \alpha_n(a)\in U\}\quad\text{and}\quad Q_P=\bigcup_{r=0}^\infty G^{\circ r}(Q_P^0).
\]
By definition, we have $Q_P\subseteq Y(\Lab)\cap\varphi^{-1}(P)$ and $G(Q_P)\subseteq Q_P$. Since $\alpha_n$, $\nu_U$, and $G$ are defined over $L$, the sets $Q_P^0$ and $Q_P$ are Galois-invariant.

Fix $\epsilon>0$ and a nonempty Zariski open subset $O\subseteq Y_{\olQ}$. Then $\nu_U(O\cap\nu^{-1}(U))$ is a nonempty Zariski open subset of $U_{\olQ}$. By \eqref{eq:Cnden}, take an index $n\in I$ such that
\[
\alpha_n^{-1}(\nu_U(O\cap\nu^{-1}(U)))
\]
is a nonempty Zariski open subset of $V_{\olQ}$. By \eqref{eq:pdense}, take $a\in P\cap \alpha_n^{-1}(\nu_U(O\cap\nu^{-1}(U)))$ with $h_{\mathrm{crit}}(a)<\epsilon/(k_dd^n)$. Then $x:=\nu_U^{-1}(\alpha_n(a))\in O\cap Q_P^0$, and \eqref{eq:cmp} shows
\[
h_{\overline{A}_Y}(x)=k_dd^nh_{\mathrm{crit}}^{\max}(a)\leq k_dd^nh_{\mathrm{crit}}(a)<\epsilon.
\]
Thus, \eqref{eq:qdense} holds.

Let $x_0=\nu_U^{-1}(\alpha_n(a))\in Q_P^0$, where $n\in I$, $a\in P$, and $\alpha_n(a)\in U$. Let $r\geq0$ be arbitrary, and set $x=G^{\circ r}(x_0)\in Q_P$. Then $\varphi(x)=a$. Since $\varphi,G,\nu_U^{-1},\alpha_n$ are defined over $L$, we obtain
\[
L(a)\subseteq L(x)\subseteq L(x_0)\subseteq L(a),
\]
and hence
\[
L(x)=L(a)=L(\varphi(x))\quad\text{in }\Lab.
\]
Since
\[
\overline{\nu}(x)=\nu(x)=(F_d^{\circ(n+rl)}\circ s_d)(a)\quad\text{and}\quad\overline{\nu}(G(x))=\nu(G(x))=(F_d^{\circ(n+(r+1)l)}\circ s_d)(a),
\]
\eqref{eq:orbitheight} implies
\[
h_{\overline{A}_Y}(G(x))=k_d d^{n+(r+1)l}h_{\mathrm{crit}}^{\max}(a)=d^lh_{\overline{A}_Y}(x).
\]
Thus, \eqref{eq:q} holds.
\end{proof}

\subsection{The parameter image of a bi-finite correspondence}
We use the assumptions, notation, and conclusion of Proposition~\ref{prop:orb}, and set
\[
b=\dim V\geq1.
\]

In the following lemma, we examine how the intertwined relation interacts with bi-finite correspondences on $\overline{Y}_{\olQ}$.
\begin{lemma}\label{lem:base}
Let $C$ be a bi-finite correspondence on $\overline{Y}_{\olQ}$ such that
\[
(G\times G)(C\cap Y_{\olQ}^2)\subseteq C\cap Y_{\olQ}^2.
\]
Let $C_1,\ldots,C_t$ be irreducible components of $C$. For $1\leq i\leq t$, define
\[
D_i=\overline{(\varphi\times\varphi)(C_i\cap Y_{\olQ}^2)}\subseteq V_{\olQ}^2.
\]

\textup{(1)} For $1\leq i\leq t$, $\dim D_i=b$, both projections $D_i\to V_{\olQ}$ are generically finite and dominant, and the generic point of $D_i$ belongs to the intertwined relation $\approx$.

\textup{(2)} Let $u_V\in L(V)\setminus L$ be as in Proposition~\ref{prop:rel}. Set $U=\varphi^*(u_V)\in L(\overline{Y})\setminus L$. For $1\leq i\leq t$, one has
\[
\pr_1^*(U)=\pr_2^*(U)\quad\text{in }\olQ(C_i),
\]
where $\pr_1,\pr_2: C_i\to\overline{Y}_{\olQ}$ are the two projections.

\textup{(3)} There is a proper closed subset $J_C\subsetneq V_{\olQ}$ such that for every $(x,y)\in(C\cap Y^2)(\C)$, either $\varphi(x)\in J_C(\C)$, or $\varphi(x)\approx\varphi(y)$.
\end{lemma}
\begin{proof}
For $1\leq i\leq t$, set $C_i^\circ=C_i\cap Y_{\olQ}^2$. Then $C_i^\circ$ is irreducible and Zariski dense in $C_i$, with $\dim C_i^\circ=\dim Y=b+s$. As $G\times G:Y_{\olQ}^2\to Y_{\olQ}^2$ is finite, we see that
\[
(G\times G)(C_i^\circ)=C_{\sigma(i)}^\circ
\]
for a unique $\sigma(i)\in\{1,\ldots,t\}$. Since $\varphi\circ G=\varphi$,
\[
D_{\sigma(i)}
=\overline{(\varphi\times\varphi)((G\times G)(C_i^\circ))}
=D_i.
\]
Thus, for every $1\leq i\leq t$, there are integers $q(i)\geq0$ and $a(i)\geq1$ such that the index
\[
j(i)=\sigma^{\circ q(i)}(i)\in\{1,\ldots,t\}
\]
satisfies $\sigma^{\circ a(i)}(j(i))=j(i)$ and $D_{j(i)}=D_i$.

\smallskip

\textup{(1)} Fix $1\leq i\leq t$. Set $j=j(i)$, $a=a(i)$, $D=D_i=D_j$, and $r_D=\dim D$. Define
\[
H=(G\times G)^{\circ a}|_{C_j^\circ}:C_{j}^{\circ}\to C_j^{\circ}.
\]
Let $\rho_1:C_j^\circ\to Y_{\olQ}$ be the restriction of the first projection to $C_j^\circ$. Then $\rho_1$ is dominant and generically finite, with $\rho_1\circ H=G^{\circ a}\circ\rho_1$. Thus, by Proposition~\ref{prop:deg},
\begin{equation}\label{eq:deg1}
\deg H=\deg G^{\circ a}=d^{asl}.
\end{equation}
In the proof of Proposition~\ref{prop:orb}, the finite normalization map
\[
\nu:Y\to Y_0\subseteq X_d\times_{\MPoly^d}V
\]
satisfies $\nu\circ G=F_d^{\circ l}|_{Y_0}\circ\nu$. Using the product form of Proposition~\ref{prop:deg}, we deduce that
\begin{equation}\label{eq:deg2}
\deg H=d^{al(\dim C_j-r_D)}=d^{al(b+s-r_D)}.
\end{equation}
From \eqref{eq:deg1} and \eqref{eq:deg2}, we obtain $r_D=b$, i.e., $\dim D=\dim V$. Clearly, both projections $D\to V$ are dominant; hence both are generically finite.

Let 
\[
\chi:A_d\times(\A^1)^{d-1}\to A_d\times S_{d-1}\to X_d
\]
be the finite surjective morphism constructed in \S~\ref{sec:2.1}. Define
\[
\widetilde{F}:A_d\times(\A^1)^{d-1}\to A_d\times(\A^1)^{d-1},\,(\ba,z_1,\ldots,z_{d-1})\mapsto(\ba,f_{\ba}(z_1),\ldots,f_{\ba}(z_{d-1})).
\]
Then $\chi\circ\widetilde{F}=F_d\circ\chi$.

Let $\Lambda=\overline{\olQ(D)}$ and $\eta=(\eta_1,\eta_2)\in D(\Lambda)$ be a geometric generic point. Since $\eta_1,\eta_2$ are generic in the positive-dimensional $V$ and the exceptional locus $E_d\subseteq\MPoly^d$ is finite, both $\eta_1$ and $\eta_2$ are non-exceptional. Write $C_{j,\eta}=C_j^\circ\times_D\Spec\Lambda$. Define
\[
\widetilde{C}_\eta=\left(C_{j,\eta}\times_{(X_d^2)_\Lambda}((A_d\times(\A^1)^{d-1})^2)_\Lambda\right)_{\red}.
\]
The projection $\widetilde{C}_\eta\to C_{j,\eta}$ is finite and surjective. The morphism $H$ and $\widetilde{F}^{\circ la}\times\widetilde{F}^{\circ la}$ induce a finite endomorphism $\widetilde{H}$ of $\widetilde{C}_\eta$. Take a periodic irreducible component $\widetilde{W}$ dominating an irreducible component $T$ of $C_{j,\eta}$, and choose $m\geq1$ with $\widetilde{H}^{\circ m}(\widetilde{W})=\widetilde{W}$. Such a component exists because the finite map sends components of dimension $s$ to components of dimension $s$. The morphism $\widetilde{W}\to T$ is finite and surjective, so $\dim\widetilde{W}=\dim T=s$.

The parameter coordinates on $\widetilde{W}$ are constant, say $\ba_k\in A_d(\Lambda)$ with $[f_{\ba_k}]=\eta_k$ for $1\leq k\leq2$. Let $W$ be the reduced image of $\widetilde{W}$ in $(\A^1_\Lambda)^{d-1}\times(\A^1_\Lambda)^{d-1}$ under the ordered-coordinate projection. Since $\widetilde{W}\to W$ is finite, we obtain $\dim W=s$. Set $f_k=f_{\ba_k}$ for $1\leq k\leq2$. Then we have
\begin{equation}\label{eq:Wper}	\left((f_1^{\circ alm})^{\times(d-1)}\times(f_2^{\circ alm})^{\times(d-1)}\right)(W)=W.
\end{equation}

For $1\leq k\leq2$, let $\rho_{\eta,k}:T\to Y_{\eta_k}$ be induced by the $k$-th projection from $C_j^\circ$, and let $\nu_{\eta,k}:Y_{\eta_k}\to(X_d)_{\eta_k}$ be induced by $\nu$. Since $\rho_{\eta,k}$ is quasi-finite and $\nu_{\eta,k}$ is finite, we obtain $\dim\nu_{\eta,k}(\rho_{\eta,k}(T))=s$. Let $q:A_d\times S_{d-1}\to X_d$ be the quotient map, and define
\[
\chi_{\ba_k}:(\A^1_\Lambda)^{d-1}\to(X_d)_{\eta_k},\,(z_1,\ldots,z_{d-1})\mapsto q\left(\ba_k,\sum_{h=1}^{d-1}[z_h]\right).
\]
For the coordinate projections $\pi_k:W\to(\A^1_\Lambda)^{d-1}$, we have
\[
\chi_{\ba_k}(\pi_k(W))=\nu_{\eta,k}(\rho_{\eta,k}(T)).
\]
Since $\chi_{\ba_k}$ is finite, we conclude that
\[
\dim W=\dim\pi_1(W)=\dim\pi_2(W)=s.
\]
Suppose that $\eta_1\not\approx\eta_2$. Then $f_1^{\circ alm}\not\approx f_2^{\circ alm}$ by definition. Since both $f_{\ba_1}$ and $f_{\ba_2}$ are non-exceptional, by \cite[Proposition~5.3(i)]{GN17} and \eqref{eq:Wper} we conclude that
\[
W=\pi_1(W)\times\pi_2(W).
\]
So
\[
s=\dim W=\dim\pi_1(W)+\dim\pi_2(W)=2s,
\]
contradicting $s>0$. Thus, $\eta_1\approx\eta_2$.

\medskip

\textup{(2)} Since $\varphi$ is dominant, $\varphi^*:L(V)\hookrightarrow L(Y)=L(\overline{Y})$ is well-defined and injective; hence $u_V\notin L$ implies $U=\varphi^*(u_V)\in L(\overline{Y})\setminus L$. Fix $1\leq i\leq t$. Let $q_1,q_2:D_i\to V_{\olQ}$ be the two projections, and let $\psi_i:C_i^\circ\to D_i$ be the restriction of $\varphi\times\varphi$. Note that $q_1,q_2,\psi_i$ are all dominant, so the pullback maps $q_1^*,q_2^*,\psi_i^*$ on function fields are well-defined. By \textup{(1)} and Proposition~\ref{prop:rel}, we obtain $q_1^*(u_V)=q_2^*(u_V)$ in $\olQ(D_i)$. Using $q_k\circ\psi_i=\varphi\circ\pr_k|_{C_i^\circ}$ ($1\leq k\leq 2$), we deduce that
\[
\pr_1^*(U)=\psi_i^*(q_1^*(u_V))=\psi_i^*(q_2^*(u_V))=\pr_2^*(U)\quad\text{in }\olQ(C_i).
\]

\medskip

\textup{(3)} Let $\sI\subseteq V_{\olQ}^2$ be the intertwined relation on $V$, which is a constructible subset of $V_{\olQ}^2$ by Proposition~\ref{prop:rel}. For each $1\leq i\leq t$, by \textup{(1)}, choose a nonempty Zariski open subset $O_i\subseteq D_i$ with $O_i\subseteq\sI$, and set $E_i=\overline{q_1(D_i\setminus O_i)}\subseteq V_{\olQ}$. Since each $D_i$ is irreducible of dimension $b$, the closed subset
\[
J_C=\bigcup_{i=1}^t E_i
\]
of $V$ has dimension
\[
\max_i\dim E_i\leq\max_i\dim(D_i\setminus O_i)<b=\dim V.
\]
So $J_C\subsetneq V_{\olQ}$. Let $(x,y)\in(C\cap Y^2)(\C)$ satisfy $\varphi(x)\notin J_C(\C)$. Take $1\leq i_0\leq t$ with $(x,y)\in C_{i_0}^\circ$. Then
\[
(\varphi(x),\varphi(y))\in O_{i_0}(\C)\subseteq\sI(\C);
\]
hence $\varphi(x)\approx\varphi(y)$.
\end{proof}

\subsection{Inertia and Frobenius}\label{sec:infr}
We finish the proof of Theorem~\ref{thm:main}.
\begin{proof}[Proof of Theorem~\ref{thm:main}]
Let $d\geq2$ be an integer, and let $K$ be a number field. If there is a constant $\epsilon_0>0$ such that
\[
\left\lbrace\alpha\in\MPoly^d(\Kab)\colon h_{\mathrm{crit}}(\alpha)<\epsilon_0\right\rbrace
\]
is finite, then for
\[
\epsilon_{K,d}=\min\left(\{\epsilon_0\}\cup\left\lbrace h_{\mathrm{crit}}(\alpha)\colon \alpha\in\MPoly^d(\Kab),\ 0<h_{\mathrm{crit}}(\alpha)<\epsilon_0\right\rbrace\right)>0,
\]
the set
\[
\left\lbrace\alpha\in\MPoly^d(\Kab)\colon h_{\mathrm{crit}}(\alpha)<\epsilon_{K,d}\right\rbrace=\left\lbrace\alpha\in\MPoly^d(\Kab)\colon h_{\mathrm{crit}}(\alpha)=0\right\rbrace
\]
is finite. Suppose, for the sake of contradiction, that for all $\epsilon>0$, the set
\[
\left\lbrace\alpha\in\MPoly^d(\Kab)\colon h_{\mathrm{crit}}(\alpha)<\epsilon\right\rbrace
\]
is infinite. Then there is a sequence $(x_j)_{j=1}^\infty\subseteq\MPoly^d(\Kab)$ of pairwise distinct points with
\[
\lim_{j\to\infty}h_{\mathrm{crit}}(x_j)=0.
\]
We use the notation and conclusion of Proposition~\ref{prop:orb}.

Take a finite subset $S$ of $M_L$ containing all archimedean places such that in Proposition~\ref{prop:orb}, one can choose a model $(\mathcal{Y},\sA_Y)$ of $(\overline{Y},\overline{A}_Y)$ over $\sO_{L,S}$, where $\sA_Y$ is relatively ample and $\mathcal{Y}$ is normal with geometrically integral special fibers.

Let $M_d$ be as in Proposition~\ref{prop:Md}. Applying the Chebotarev density theorem \cite[\S~\Rom{7}.13]{Neukirch1999} to the Galois closure of $L/\Q$ in $\olQ$, we can take a rational prime $p>\max\{2,M_d\}$ splitting completely in $L$ and a place $v\in M_L\setminus S$ lying over $p$. Then $L_v=\Q_p$ with the residue field $\F_p$. Fix an extension $\bar{v}$ of $v$ to $\olQ$. Define
\begin{align*}
J_{\mathrm{w}}&=\left\lbrace j\in\Z_{>0}\colon L(a_j)_w/\Q_p\text{ is wildly ramified, where }w=\bar{v}|_{L(a_j)}\right\rbrace,\\
J_{\mathrm{t}}&=\left\lbrace j\in\Z_{>0}\colon L(a_j)_w/\Q_p\text{ is tamely ramified, where }w=\bar{v}|_{L(a_j)}\right\rbrace=\Z_{>0}\setminus J_{\mathrm{w}}.
\end{align*}
Then at least one of $J_{\mathrm{w}}$ and $J_{\mathrm{t}}$ is infinite. Define
\[
P_{\mathrm{w}}=\Gal(\Lab/L)\cdot\{a_j\colon j\in J_{\mathrm{w}}\}\quad\text{and}\quad P_{\mathrm{t}}=\Gal(\Lab/L)\cdot\{a_j\colon j\in J_{\mathrm{t}}\}.
\]
Then
\[
P:=\Gal(\Lab/L)\cdot\{a_j\colon j\geq1\}=P_{\mathrm{w}}\sqcup P_{\mathrm{t}}.
\]
Since $\Gal(\Lab/L)$ is abelian, for every $\alpha\in P$ with $w=\bar{v}|_{L(\alpha)}$, the extension $L(\alpha)_w/\Q_p$ is wildly ramified if $\alpha\in P_{\mathrm{w}}$, and it is tamely ramified if $\alpha\in P_{\mathrm{t}}$. Set $\delta=p^{-1/(p-1)}\in(0,1)$.

\medskip

\textbf{Case 1.} Suppose that $J_{\mathrm{w}}$ is infinite.

Then $\{a_j\colon j\in J_{\mathrm{w}}\}$ is generic in $V/\olQ$, and the subset $P_{\mathrm{w}}\subseteq V(\Lab)$ is $\Gal(\Lab/L)$-invariant and satisfies \eqref{eq:pdense}. Let $Q_{\mathrm{w}}=Q_{P_{\mathrm{w}}}$ be given by Proposition~\ref{prop:orb}.

Define
\[
\mathcal{E}=\{L(x)\colon x\in Q_{\mathrm{w}}\}.
\]
Let $E\in\mathcal{E}$ be arbitrary, and denote $w=\bar{v}|_E$. Fix $x\in Q_{\mathrm{w}}$ with $E=L(x)=L(\varphi(x))$. Under the canonical isomorphism
\[
D_w(E/L)\simeq\Gal(E_w/\Q_p),
\]
the inertia subgroup $I_w(E/L)$ corresponds to $I(E_w/\Q_p)$. By Lemma~\ref{lem:in}(1), there exists $\tau_E\in I_w(E/L)$ of order $p$ such that
\begin{equation}\label{eq:inest}
|\tau_E(z)-z|_w\leq\delta\quad(z\in\sO_{E_w}).
\end{equation}
For every field $E\in\mathcal{E}$, fix one such $\tau_E\in\Gal(E/L)$, and an extension $\tau_E\in\Gal(\Lab/L)$ still denoted by $\tau_E$. Define
\[
B_{\mathrm{In}}=\left\lbrace(x,\tau_{L(x)}(x))\colon x\in Q_{\mathrm{w}}\right\rbrace\subseteq Q_{\mathrm{w}}^2\subseteq Y(\Lab)^2.
\]
Since $\Gal(\Lab/L)$ is abelian, $B_{\mathrm{In}}$ is $\Gal(\Lab/L)$-invariant, and hence $\Gal(\olQ/L)$-invariant. From \eqref{eq:q} and $\varphi\circ G=\varphi$, for $x\in Q_{\mathrm{w}}$ we have
\[
L(G(x))=L(\varphi(x))=L(x)\quad\text{and}\quad G(\tau_{L(x)}(x))=\tau_{L(G(x))}(G(x)).
\]
So
\begin{equation}\label{eq:BIninv}
(G\times G)(B_{\mathrm{In}})\subseteq B_{\mathrm{In}}.
\end{equation}

Use the projective coordinates defining the tube at $v$ as in \S~\ref{sec:tube}. For
\[
x=[x_0:\cdots:x_N]\in Q_{\mathrm{w}}
\]
with $E=L(x)$ and $w=\bar{v}|_E$, choose $0\leq i\leq N$ with $|x_i|_w=\max_j|x_j|_w$, and set $z_j=x_j/x_i\in\sO_{E_w}$. From \eqref{eq:inest}, we obtain
\[
\left|\frac{\tau_E(x_j)}{\tau_E(x_i)}-\frac{x_j}{x_i}\right|_w
=|\tau_E(z_j)-z_j|_w\leq\delta.
\]
So $B_{\mathrm{In}}\subseteq T_1(\delta)$.

Let
\[
Z_{\mathrm{In}}=Z(B_{\mathrm{In}})=\bigcap_{n=1}^\infty\overline{(B_{\mathrm{In}})_{1/n}}\subseteq\overline{Y}_{\olQ}^2
\]
be as in Proposition~\ref{prop:tube}. By Noetherianity, there is an integer $n_0\geq1$ such that $Z_{\mathrm{In}}=\overline{(B_{\mathrm{In}})_{1/n_0}}$ is Zariski closed in $\overline{Y}_{\olQ}^2$. Let
\[
\pr_1,\pr_2:\overline{Y}_{\olQ}^2\to\overline{Y}_{\olQ}
\]
be the two projections. By \eqref{eq:qdense} and the construction of $Z_{\mathrm{In}}$, we see that some irreducible component $C'$ of $Z_{\mathrm{In}}$ dominates the first factor $\overline{Y}_{\olQ}$; let $C'$ be such a component. Assume that $\pr_2(C')=H\subsetneq\overline{Y}_{\olQ}$. Define
\[
H_L=\bigcup_{\gamma\in\Gal(\olQ/L)}\gamma(H).
\]
Then $H_L\subsetneq\overline{Y}_{\olQ}$ is a proper closed $\Gal(\olQ/L)$-invariant subset because $H$ is defined over a finite extension of $L$. For every $(x,\tau_{L(x)}(x))\in B_{\mathrm{In}}\cap C'$, we have $\tau_{L(x)}(x)\in H\subseteq H_L$, and then $x\in H_L$. We deduce that $\pr_1(C')\subseteq H_L$, contradicting the dominance of $\pr_1|_{C'}$. Thus, the second projection $\pr_2|_{C'}:C'\to\overline{Y}_{\olQ}$ is also dominant. By Proposition~\ref{prop:tube}, both projections $\pr_1|_{C'},\pr_2|_{C'}:C'\to\overline{Y}_{\olQ}$ are finite and surjective.

Using \eqref{eq:BIninv}, we see that there are an integer $m\geq1$ and irreducible components $C_1',\ldots,C_m'$ of $Z_{\mathrm{In}}$ whose two projections are finite and surjective such that
\[
\overline{(G\times G)(C_j'\cap Y_{\olQ}^2)}=C_{j+1}'\quad\text{in }\overline{Y}_{\olQ}^2
\]
for $j$ modulo $m$. Let
\[
C_{\mathrm{In}}=\bigcup_{j=1}^{m}C_j'
\]
with the reduced structure. Then $C_{\mathrm{In}}$ is a bi-finite correspondence on $\overline{Y}_{\olQ}$ and
\[
(G\times G)(C_{\mathrm{In}}\cap Y_{\olQ}^2)\subseteq C_{\mathrm{In}}\cap Y_{\olQ}^2.
\]

Let $J_{C_{\mathrm{In}}}\subsetneq V$ be given by Lemma~\ref{lem:base}(3). Clearly, there is $x\in Q_{\mathrm{w}}$ such that
\[
(x,\tau_{L(x)}(x))\in C_{\mathrm{In}}\quad\text{and}\quad\varphi(x)\notin J_{C_{\mathrm{In}}}.
\]
Set $E=L(x)=L(\varphi(x))$. By Lemma~\ref{lem:base}(3), we obtain
\[
a:=\varphi(x)\approx\varphi(\tau_E(x))=\tau_E(a).
\]
Since $\approx$ is defined over $\Q$, applying $\tau_E^j$ for $0\leq j<p$ implies
\[
a\approx\tau_E(a)\approx\cdots\approx\tau_E^{p-1}(a).
\]
Since $E=L(a)$ and $\tau_E$ has order $p$ in $\Gal(E/L)$, the $p$ points
\[
a,\tau_E(a),\ldots,\tau_E^{p-1}(a)\in V(\Lab)
\]
are pairwise distinct and intertwined, contradicting $p>M_d$.

\medskip

\textbf{Case 2.} Suppose that $J_{\mathrm{t}}$ is infinite.

Then $\{a_j\colon j\in J_{\mathrm{t}}\}$ is generic in $V/\olQ$, and the subset $P_{\mathrm{t}}\subseteq V(\Lab)$ is $\Gal(\Lab/L)$-invariant and satisfies \eqref{eq:pdense}. Let $Q_{\mathrm{t}}=Q_{P_{\mathrm{t}}}$ be given by Proposition~\ref{prop:orb}.

For $a\in P_{\mathrm{t}}$ with $w=\bar{v}|_{L(a)}$, Lemma~\ref{lem:in}(2) shows that the ramification index
\[
e(L(a)_w/\Q_p)\mid(p-1).
\]
For $x\in Q_{\mathrm{t}}$, by Proposition~\ref{prop:orb} we have $a=\varphi(x)\in P_{\mathrm{t}}$ and $L(x)=L(a)$, hence
\[
e(L(x)_w/\Q_p)\leq p-1.
\]
Take $\sigma\in\Gal(\olQ/L)$ in the decomposition group at the chosen place $\bar{v}$ of $\olQ$, inducing $z\mapsto z^p$ on the residue field $\overline{\F}_p$, and also denote the restriction $\sigma|_{\Lab}$ by $\sigma$. Define
\[
B_{\mathrm{Fr}}=\left\lbrace(x,\sigma(x))\colon x\in Q_{\mathrm{t}}\right\rbrace.
\]
Since $\Gal(\Lab/L)$ is abelian, $B_{\mathrm{Fr}}$ is $\Gal(\Lab/L)$-invariant, and hence $\Gal(\olQ/L)$-invariant. Since $G(Q_{\mathrm{t}})\subseteq Q_{\mathrm{t}}$ and $G$ is defined over $L$, we obtain
\begin{equation}\label{eq:BFrinv}
(G\times G)(B_{\mathrm{Fr}})\subseteq B_{\mathrm{Fr}}.
\end{equation}
Let $x\in Q_{\mathrm{t}}$ with $w=\bar{v}|_{L(x)}$. For $z\in\sO_{L(x)_w}$, the choice of $\sigma$ implies $\sigma(z)-z^p\in\mathfrak{m}_{L(x)_w}$, and then
\[
|\sigma(z)-z^p|_w\leq p^{-1/e(L(x)_w/\Q_p)}\leq p^{-1/(p-1)}=\delta.
\]
Choosing a maximal projective coordinate of $x$ as in Case 1 shows
\[
(x,\sigma(x))\in T_p(\delta).
\]

Let $Z_{\mathrm{Fr}}=Z(B_{\mathrm{Fr}})$ be as in Proposition~\ref{prop:tube}, which is Zariski closed in $\overline{Y}_{\olQ}^2$. Let $\pr_1,\pr_2:\overline{Y}_{\olQ}^2\to\overline{Y}_{\olQ}$ be the two projections. As in Case 1, some irreducible component $C'$ of $Z_{\mathrm{Fr}}$ dominates the first factor $\overline{Y}_{\olQ}$; for such a component $C'$, both projections $\pr_1|_{C'},\pr_2|_{C'}:C'\to\overline{Y}_{\olQ}$ are finite and surjective.

Using \eqref{eq:BFrinv}, we see that there are an integer $m\geq1$ and irreducible components $C_1',\ldots,C_m'$ of $Z_{\mathrm{Fr}}$ whose two projections are finite and surjective such that
\[
\overline{(G\times G)(C_j'\cap Y_{\olQ}^2)}=C_{j+1}'\quad\text{in }\overline{Y}_{\olQ}^2
\]
for $j$ modulo $m$. Let
\[
C_{\mathrm{Fr}}=\bigcup_{j=1}^{m}C_j'
\]
with the reduced structure. Then $C_{\mathrm{Fr}}$ is a bi-finite correspondence on $\overline{Y}_{\olQ}$ and
\[
(G\times G)(C_{\mathrm{Fr}}\cap Y_{\olQ}^2)\subseteq C_{\mathrm{Fr}}\cap Y_{\olQ}^2.
\]
Let $\nu:W\to C_1'$ be the normalization, and set $\pi_k=\pr_k|_{C_1'}\circ\nu:W\to\overline{Y}_{\olQ}$ for $1\leq k\leq2$. Applying Proposition~\ref{prop:tube}, we obtain
\[
\pi_2^*A_{Y,\olQ}\equiv p\pi_1^*A_{Y,\olQ}.
\]
Let $u_V\in L(V)\setminus L$ be as in Proposition~\ref{prop:rel}. By Lemma~\ref{lem:base}, the nonconstant rational function $U=\varphi^*(u_V)\in L(\overline{Y})\setminus L$ satisfies $\pi_1^*(U)=\pi_2^*(U)$ in $\olQ(W)$. Take a finite extension $L'/L$ over which $C_1'$ is defined; then $W,\nu,\pi_1,\pi_2$ are also defined over $L'$. We have $U\in L'(\overline{Y}_{L'})\setminus L'$ and $\pi_1^*(U)=\pi_2^*(U)$ in $L'(W)$. Thus, applying Lemma~\ref{lem:inv} over $L'$ gives a contradiction.

\smallskip

We obtain a contradiction in both cases. Therefore, there exists a constant $\epsilon_{K,d}>0$ such that
\[
\left\lbrace\alpha\in\MPoly^d(\Kab)\colon h_{\mathrm{crit}}(\alpha)<\epsilon_{K,d}\right\rbrace=\left\lbrace\alpha\in\MPoly^d(\Kab)\colon h_{\mathrm{crit}}(\alpha)=0\right\rbrace
\]
is finite.
\end{proof}

\section{Specializations of a marked point}\label{sec:app}
In this section, we prove Theorem~\ref{thm:pair}. The argument is similar to that for Theorem~\ref{thm:main}, and we omit repetitive details for brevity.

\begin{proof}[Proof of Theorem~\ref{thm:pair}]
For every finite extension $K_2/K_1$ of number fields, one has $\Kab_1\subseteq\Kab_2$. Thus, we may replace $K$ with a finite extension whenever necessary.

Write
\[
f_t(z)=A_d(t)z^d+A_{d-1}(t)z^{d-1}+\cdots+A_0(t),\quad t\in B(\olQ),
\]
where $A_d,\ldots,A_0\in K[B]$ are regular functions on $B$, and $A_d$ is everywhere non-vanishing on $B(\olQ)$. In particular, $A_d\in K[B]^\times$.

It suffices to consider the smooth case. After deleting finitely many points of $B$ and replacing $K$ with a finite extension if necessary, we may assume that $B$ is smooth and it embeds into a geometrically integral smooth projective curve $\overline{B}$ defined over $K$ such that
\[
S_\infty:=(\overline{B}\setminus B)(\olQ)=(\overline{B}\setminus B)(K).
\]

The proof is then divided into four steps.

\medskip

\textbf{Step 1. Function field heights.}

Let
\[
F=\olQ(\overline{B})=\olQ(B)
\]
be the function field of $\overline{B}$ (and $B$) over $\olQ$. We use the following conventions on the height functions over $F$, and refer the reader to \cite[1.4.6]{HtBook} and \cite{Baker2009} for more information. The set $M_F$ of places of $F$ is identified with $\overline{B}(\olQ)$. For every $x\in\overline{B}(\olQ)$, the corresponding absolute value $|\cdot|_x$ on $F$ is determined by
\[
|f|_x=\exp(-\operatorname{ord}_x(f)).
\]
Then these absolute values $(|\cdot|_x)_{x\in M_F}$ satisfy the product formula and give the standard Weil height function
\[
h_F:\P^1(F)\to\R_{\geq0},\,[z_0:z_1]\mapsto\sum_{x\in M_F}\log\max\left\lbrace|z_0|_x,|z_1|_x\right\rbrace.
\]
Observe that the constant field of $F$ is $\olQ$. The canonical height $\hat{h}_{f,F}$ of $f$ on $\P^1(F)$ is determined by
\[
\hat{h}_{f,F}(\lambda)=\lim_{n\to\infty}\frac{1}{d^n}h_F(f^{\circ n}(\lambda)).
\]
We identify $F$ with the affine chart $\{Z_1=1\}\subseteq\P^1(F)$, where $[Z_0:Z_1]$ are homogeneous coordinates on $\P^1$. The functions $h_F$ and $\hat{h}_{f,F}$ extend to $\P^1(\overline{F})$; see \cite{Baker2009}.

Since $\varphi$ is nonconstant, $f$ is non-isotrivial over $F$. By Baker \cite[Corollary~1.8]{Baker2009},
\[
\hat{h}_{f,F}(a)>0
\]
because $a$ is not $f$-preperiodic.

For $c\in S_\infty$, define
\begin{align*}
\nu_c&=\operatorname{ord}_c(A_d),\\
s_c&=\max\left(\left\{0,\frac{\nu_c}{d-1}\right\}\cup
\left\{\frac{\nu_c-\operatorname{ord}_c(A_j)}{d-j}:0\leq j<d,\ A_j\neq0\right\}\right)\in\Q_{\geq0},\\
\lambda_c&=\lim_{m\to\infty}\frac{1}{d^m}\max\{0,-\operatorname{ord}_c(f^{\circ m}(a))\}.
\end{align*}
We show that $\lambda_c$ is well-defined. If $-\operatorname{ord}_c(f^{\circ n}(a))\leq s_c$ for all $n\geq0$, then $\lambda_c=0$. Suppose now that $-\operatorname{ord}_c(f^{\circ n}(a))>s_c$ for some $n\geq0$. Then inductive computations show
\begin{equation}\label{eq:nrord}
-\operatorname{ord}_c(f^{\circ(n+r)}(a))=-d^r\operatorname{ord}_c(f^{\circ n}(a))-\frac{d^r-1}{d-1}\nu_c
>s_c\quad(r\geq0),
\end{equation}
and then
\[
\lambda_c=\lim_{r\to\infty}\left(-\frac{\operatorname{ord}_c(f^{\circ n}(a))}{d^n}-\frac{d^r-1}{d^{n+r}(d-1)}\nu_c\right)=\frac{1}{d^n}\left(-\operatorname{ord}_c(f^{\circ n}(a))-\frac{\nu_c}{d-1}\right)>0.
\]
Thus, the limit defining $\lambda_c$ exists and $(d-1)\lambda_c\in\Z[1/d]\cap\Q_{\geq0}$.

Define
\[
S_+=\{c\in S_\infty:\lambda_c>0\}.
\]
Since $a$ and the coefficients of $f$ are regular on $B$, by the local decomposition of $\hat{h}_{f,F}$ (see \cite[\S~1.2]{Baker2009}) we compute that
\[
\sum_{c\in S_+}\lambda_c=\sum_{c\in S_\infty}\lambda_c=\hat{h}_{f,F}(a)>0.
\]
In particular, $S_+\neq\emptyset$.

Fix a sufficiently large integer $N$ such that for $b=f^{\circ N}(a)$, one has
\[
\ell_c:=-\operatorname{ord}_c(b)>s_c,\quad
k_c:=(d-1)\ell_c-\nu_c=(d-1)d^N\lambda_c\in\Z_{>0}
\]
for all $c\in S_+$, and
\[
\sum_{c\in S_+}k_c=(d-1)d^N\hat{h}_{f,F}(a)\geq2g(\overline{B}).
\]
For $r\geq0$, set $b_r=f^{\circ r}(b)$. Define
\[
D=\sum_{c\in S_+}k_c[c]\neq0,
\]
which is an effective divisor on $\overline{B}$.

For each $c\in S_\infty$, choose a uniformizer $w_c\in K(\overline{B})$, i.e., $\operatorname{ord}_c(w_c)=1$. For $c\in S_+$, set
\[
\beta_c=(w_c^{k_c}A_d b^{d-1})(c)\in K^\times.
\]
For $c\in S_+$ and $r\geq0$, using $\ell_c>s_c$ and \eqref{eq:nrord} we compute inductively that
\begin{equation}\label{eq:bdeg1}
-\operatorname{ord}_c(b_r)=d^r\ell_c-\frac{d^r-1}{d-1}\nu_c\quad\text{and}\quad u_r(c)=\beta_c^{d^r},
\end{equation}
where $u_r:=w_c^{k_cd^r}A_d b_r^{d-1}$. For $c\in S_\infty\setminus S_+$, we have
\begin{equation}\label{eq:bdeg2}
\max\{-\operatorname{ord}_c(b_r)\colon r\geq0\}\leq s_c.
\end{equation}

\medskip

\textbf{Step 2. A nef adelic line bundle.}

Define $H=\sO_{\overline{B}}(D)$, which is an ample line bundle on $\overline{B}$. In this step, we construct a nef adelic line bundle $\overline{H}$ with underlying line bundle $H$ such that
\begin{equation}\label{eq:pairheight}
h_{\overline{H}}(t)=(d-1)\hat{h}_{f_t}(b(t))
\end{equation}
for all $t\in B(\olQ)$. Here $h_{\overline{H}}$ is the absolute logarithmic height of $\overline{H}$, normalized over $\Q$. This result is an analogue of Theorem~\ref{thm:met}.

Let $v\in M_K$. Define
\[
\psi_v(t)=(d-1)\lim_{r\to\infty}\frac{1}{d^r}\log\max\left\lbrace1,|b_r(t)|_v\right\rbrace,\quad t\in B(\C_v).
\]
By \cite[Theorem~4.27]{FG22}, it extends to a nonnegative continuous subharmonic function on $B_v^{\mathrm{an}}$. Regard $1$ as the canonical rational section of $H$ with divisor $D$, and define its norm at $v$ by $-\log\|1\|_v=\psi_v$. By \eqref{eq:bdeg1} and \eqref{eq:bdeg2}, the boundary coefficient in \cite[Theorem~4.27]{FG22} is
\[
q_c(f,b)=\lim_{r\to\infty}\frac{1}{d^r}\max\{0,-\operatorname{ord}_c(b_r)\}
=\frac{\operatorname{ord}_c(D)}{d-1};
\]
hence, \cite[Theorem~4.27]{FG22} shows that for every $c\in S_\infty$, the function
\[
\psi_v(t)-\operatorname{ord}_c(D)\log|w_c(t)|_v^{-1}
\]
extends continuously and subharmonically through $c$. When $c\in S_+$, using \eqref{eq:bdeg1} on a sufficiently small punctured analytic disk at $c$, it is easy to compute that
\begin{equation}\label{eq:pairboundary}
\lim_{t\to c}\left(\psi_v(t)-k_c\log|w_c(t)|_v^{-1}\right)=\log|\beta_c|_v.
\end{equation}

By \cite[Theorem~4.28]{FG22}, the metric $\|\cdot\|_v$ is semipositive for every $v\in M_K$. Define
\[
\|\cdot\|=(\|\cdot\|_v)_{v\in M_K}\quad\text{and}\quad\overline{H}=(H,\|\cdot\|).
\]
We show that $\overline{H}$ satisfies the property $(*)$: There are a finite subset $S\subseteq M_K$ containing all archimedean places, a smooth projective model $\mathcal{B}$ of $\overline{B}$ over $\sO_{K,S}$ with geometrically integral fibers, and a relatively ample model $\sH$ of $H$ on $\mathcal{B}$ such that $\|\cdot\|_v$ is induced by $(\mathcal{B},\sH)$ for all $v\notin S$.

By Riemann--Roch and $\deg D\geq2g(\overline{B})$, we can choose
\[
\eta\in H^0(\overline{B},\sO(D))\setminus\bigcup_{c\in S_+}H^0(\overline{B},\sO(D-[c])).
\]
Viewed as a rational function, $\eta\in K(\overline{B})$ has pole divisor exactly $D$. Thus, the morphism $\eta:\overline{B}\to\P^1_K$ is finite of degree $\deg D$, and
\[
\eta^*\sO_{\P^1_K}(1)\simeq H.
\]
Set
\[
\mathcal{F}=\{\eta,b\}\cup\{A_j\colon A_j\neq0,\ 0\leq j\leq d\}
\cup\{w_c\colon c\in S_\infty\}\subseteq K(\overline{B}).
\]
Let $T\subseteq\overline{B}$ be the finite reduced closed subscheme consisting of $S_\infty$ and the supports of the zero and pole divisors of the functions in $\mathcal{F}$. By spreading out, there is a finite subset $S\subseteq M_K$ containing all archimedean places and a smooth projective $\sO_{K,S}$-model $\mathcal{B}$ of $\overline{B}$, with geometrically integral fibers, such that $\eta$ extends to a finite morphism
\[
\boldsymbol{\eta}:\mathcal{B}\to\P^1_{\sO_{K,S}}.
\]
Define
\[
\sH:=\boldsymbol{\eta}^*\sO_{\P^1_{\sO_{K,S}}}(1),
\]
which is a line bundle on $\mathcal{B}$ with generic fiber $\eta^*\sO_{\P^1_K}(1)\simeq H$. The line bundle $\sH$ is relatively ample because $\boldsymbol{\eta}$ is finite. After enlarging $S$, we may also assume that the following conditions hold:
\begin{itemize}
\item The Zariski closure of $T$ in $\mathcal{B}$ is finite \'etale over $\sO_{K,S}$.
\item If $x_1,x_2$ are two distinct closed points of $T$, then the Zariski closures of $x_1,x_2$ in $\mathcal{B}$ are disjoint.
\item For every $h\in\mathcal{F}$ with $\operatorname{div}_{\overline{B}}(h)=\sum_{j}\mu_j[y_j]$, one has $\operatorname{div}_{\mathcal{B}}(h)=\sum_{j}\mu_j[\overline{y_j}]$.
\end{itemize}
Note that each $c\in S_\infty$ extends to a section $\mathcal{C}_c:\Spec\sO_{K,S}\to\mathcal{B}$, and near $\mathcal{C}_c$ the function $w_c$ generates its ideal. Moreover,
\[
\boldsymbol{\eta}^*(\infty_R)=\sum_{c\in S_+}k_c[\mathcal{C}_c],
\]
where $\infty_R$ is the $\sO_{K,S}$-point at infinity of $\P^1_{\sO_{K,S}}$. For every $h\in\mathcal{F}$ and every $c\in S_\infty$, on a Zariski open neighborhood of $\mathcal{C}_c$ we have
\begin{equation}\label{eq:uhc}
h=w_c^{\operatorname{ord}_c(h)}u_{h,c},
\end{equation}
where $u_{h,c}\in\sO_{\mathcal{B}}^\times$.

Fix $v\in M_K\setminus S$. We show that $\|\cdot\|_v$ is induced by $(\mathcal{B},\sH)$. The set
\[
U_v=\left\lbrace t\in B_v^{\mathrm{an}}\colon|\eta(t)|_v>1\right\rbrace
\]
is the union of the residue disks of the points of $S_+$, with their centers removed. On the punctured residue disk $D_c^*\subseteq B_v^{\mathrm{an}}$ at $c\in S_+$, from \eqref{eq:uhc} we deduce that
\[
|b(t)|_v=\rho(t)^{-\ell_c},\quad|A_d(t)|_v=\rho(t)^{\nu_c},\quad|\eta(t)|_v=\rho(t)^{-k_c}=|A_d(t)|_v|b(t)|_v^{d-1},
\]
where $0<\rho(t):=|w_c(t)|_v<1$. For $c\in S_+$, using the inequality $\ell_c>s_c$, we compute inductively that
\[
|b_r(t)|_v=|A_d(t)|_v^{(d^r-1)/(d-1)}|b(t)|_v^{d^r}\quad(r\geq1)\quad\text{on }D_c^*.
\]
Taking logarithms and letting $r\to\infty$ yield
\begin{equation}\label{eq:psivUv}
\psi_v(t)=(d-1)\log|b(t)|_v+\log|A_d(t)|_v=\log|\eta(t)|_v,\quad t\in U_v.
\end{equation}
Set $B_0=\overline{B}\setminus\Supp(D)=B\cup(S_\infty\setminus S_+)$. By \cite[Theorem~4.27]{FG22}, $\psi_v$ extends to a nonnegative continuous subharmonic function on $(B_0)_v^{\mathrm{an}}$. Then
\[
C_v=\left\lbrace t\in\overline{B}_v^{\mathrm{an}}\colon|\eta(t)|_v\leq1\right\rbrace\subseteq(B_0)_v^{\mathrm{an}}
\]
is a compact affinoid domain, with boundary equaling the boundary of $U_v$ in $(B_0)_v^{\mathrm{an}}$. Observe that the boundary values of $\psi_v$ are zero by \eqref{eq:psivUv}. By the non-negativity of $\psi_v$ and the maximum principle, $\psi_v=0$ on $C_v$. Combined with \eqref{eq:psivUv}, we obtain
\[
\psi_v(t)=\log\max\{1,|\eta(t)|_v\},\quad t\in B_v^{\mathrm{an}}.
\]
Via the identification $H\simeq\eta^*\sO(1)$, we see that $\|\cdot\|_v$ is the model metric of $(\mathcal{B},\sH)$ at $v$.

The proof of property $(*)$ is complete. From $(*)$, we deduce that $\|\cdot\|$ is an adelic metric; see \cite[\S~A.5.2]{YZ26}.

\smallskip

Let $t\in B(\olQ)$, and take a finite extension $E/K$ with $t\in B(E)$. Using the canonical section $1$, by \cite[\S~A.5.3]{YZ26} and the local decomposition of the canonical height we compute
\[
h_{\overline{H}}(t)=\sum_{w\in M_E}\frac{[E_w:\Q_w]}{[E:\Q]}\psi_w(t)
=(d-1)\hat{h}_{f_t}(b(t))\geq0,
\]
where $\psi_w$ is obtained by base change from the corresponding place of $K$. Thus, \eqref{eq:pairheight} holds.

Finally, we check the nefness of $\overline{H}$. For $c\in S_+$, the section $w_c^{-k_c}$ is a nonvanishing local frame of $H$ at $c$, and \eqref{eq:pairboundary} shows $-\log\|w_c^{-k_c}(c)\|_v=\log|\beta_c|_v$. Since $\beta_c\in K^\times$, the product formula yields
\[
h_{\overline{H}}(c)=-\sum_{v\in M_K}n_v\log\|w_c^{-k_c}(c)\|_v=\sum_{v\in M_K}n_v\log|\beta_c|_v=0.
\]
For $c\in S_\infty\setminus S_+$, computation using the local frame $1$ gives
\[
h_{\overline{H}}(c)=\sum_{v\in M_K}n_v\psi_v(c)\geq0.
\]
The adelic line bundle $\overline{H}$ has semipositive metrics at all places and nonnegative heights at all points of $\overline{B}(\olQ)$, so it is nef.

\medskip

\textbf{Step 3. The intertwined relation and bi-finite correspondences.}

We prove analogues of Proposition~\ref{prop:Md} and Lemma~\ref{lem:base} in this step.

Let $V$ be the Zariski closure of $\varphi(B)$ in $\MPoly^d_K$, which is a geometrically integral affine curve. Define an equivalence relation $\sim$ on $B(\olQ)$ by
\[
t\sim t'\quad\Longleftrightarrow\quad f_t\approx f_{t'}.
\]
We show that there is an integer $N_0\geq1$ such that every $\sim$-class has cardinality at most $N_0$. Since $\varphi$ is nonconstant and $\MPoly^d$ is affine, we can take a regular function $q$ on $\MPoly^d$ over $K$ such that $q\circ\varphi\in K[B]$ is nonconstant. The regular function $q\circ\varphi$ extends to a finite morphism $\overline{B}\to\P^1$ of degree $e_0\geq1$. Then every fiber of $\varphi$ over a point in $V(\olQ)$ has at most $e_0$ points. Let $M_d\in\Z_{>0}$ be as in Proposition~\ref{prop:Md}. Set $N_0=e_0M_d$. Then every $\sim$-class in $B(\olQ)$ has cardinality at most $N_0$. By Proposition~\ref{prop:rel}(1), the relation $\sim$ is a constructible subset of $B^2(\olQ)$ defined over $K$.

Applying Proposition~\ref{prop:rel}(2) to $V/K$, we obtain a nonconstant function
\[
u_V\in K(V)\setminus K
\]
satisfying the conclusion of Proposition~\ref{prop:rel}(2). Define
\[
u=\varphi^*(u_V)\in K(B)\setminus K.
\]

\smallskip

We prove the following property $(\Delta)$: Let $C\subseteq\overline{B}_{\olQ}^2$ be an irreducible bi-finite correspondence such that
\begin{equation}\label{eq:paircorsmall}
\overline{\{(t,t')\in C\cap B(\olQ)^2:h_{\overline{H}}(t)+h_{\overline{H}}(t')<\epsilon\}}=C,\quad\forall\epsilon>0.
\end{equation}
Then
\begin{enumerate}
\item[(i)] $\pr_1^*(u)=\pr_2^*(u)$, where $\pr_1,\pr_2:C\to\overline{B}_{\olQ}$ are the two projections.
\item[(ii)] There is a finite subset $J_C\subseteq B(\olQ)$ such that
\begin{equation}\label{eq:JCprop}
(t,t')\in C\cap B(\olQ)^2,\quad t\notin J_C\quad\Longrightarrow\quad t\sim t'.
\end{equation}
\end{enumerate}

\smallskip

Suppose that we are in the setting of $(\Delta)$. Then $C$ is an irreducible curve. We first show that the generic pair of $(\varphi\times\varphi)(C)$ is intertwined.

Take a number field $L\supseteq K$ over which $C$ is defined. Let $W\to C$ be the normalization, and write $\pi_1,\pi_2:W\to\overline{B}_{\olQ}$ for the two projections. Define
\[
W^\circ=\pi_1^{-1}(B_{\olQ})\cap\pi_2^{-1}(B_{\olQ}).
\]
Let $1\leq i\leq2$. Define
\[
\overline{H}_i=\pi_i^*\overline{H},\quad e_i=\deg\pi_i\geq1,\quad
(f_i,b_i)=\left((\pi_i|_{W^\circ})^*(f),(\pi_i|_{W^\circ})^*(b)\right).
\]
Then $f_i$ is a family of degree-$d$ polynomials parametrized by $W^\circ$, and $b_i$ is a marked point of $f_i$. Since $\varphi\circ\left(\pi_i|_{W^\circ}\right)$ is nonconstant and $b_i$ is not $f_i$-preperiodic, $(f_i,b_i)$ is an active dynamical pair parametrized by $W^\circ$, in the sense of \cite{FG22}. Observe that $(f_i,b_i)$ is non-integrable in the sense of \cite{FG22}, because the exceptional locus $E_d$ is finite and $\varphi\circ\left(\pi_i|_{W^\circ}\right)$ is nonconstant.

For $w\in W^\circ(\olQ)$ and $1\leq i\leq2$, define
\[
h_i(w):=(d-1)\hat{h}_{(f_i)_w}(b_i(w))=h_{\overline{H}_i}(w),
\]
where the last equality follows from Step 2. By \eqref{eq:paircorsmall}, we can take a sequence $(w_j)_{j=1}^\infty\subseteq W^\circ(\olQ)$ of pairwise distinct points such that
\[
\lim_{j\to\infty}\left(h_1(w_j)+h_2(w_j)\right)=0.
\]
As in the proof of Proposition~\ref{prop:tube}, the fundamental inequality implies
\[
h_{\overline{H}_i}(W)=0,\quad1\leq i\leq2.
\]
Applying Yuan's equidistribution theorem \cite[Theorem~3.1]{Yuan2008} twice, we obtain
\begin{equation}\label{eq:pairmeasures}
\frac{c_1(\overline{H}_1)_v}{e_1}=\frac{c_1(\overline{H}_2)_v}{e_2}
\end{equation}
for all $v\in M_L$.

For $v\in M_L$ and $t\in(W^\circ)_v^{\mathrm{an}}$, define
\[
g_{i,v}(t)=\psi_v(\pi_i(t))\quad(1\leq i\leq2),
\]
where the metrics and potentials are base changed from $K$ to $L$. Using the sections $s_i=\pi_i^*(1)$ of $\overline{H}_i$ ($1\leq i\leq2$), from \eqref{eq:pairmeasures} we obtain
\begin{equation}\label{eq:Del12}
e_2\Delta g_{1,v}=e_1\Delta g_{2,v}\quad\text{on }(W^\circ)_v^{\mathrm{an}}.
\end{equation}

Let $w\in M_L$ be an arbitrary archimedean place. By \cite[Theorem~5.7]{FG22} and \eqref{eq:Del12}, we have
\begin{equation}\label{eq:gvarch}
e_2g_{1,w}=e_1g_{2,w}\quad\text{on }(W^\circ)_w^{\mathrm{an}}.
\end{equation}
Let $q\in\left(W\setminus W^\circ\right)(\olQ)$ be arbitrary. Set $m_{i,q}=\operatorname{ord}_q(\pi_i^*(D))$ ($1\leq i\leq2$), and choose a local parameter $z$ at $q$. The boundary expansion of $\psi_w$ shows
\[
g_{i,w}=m_{i,q}\log|z|_w^{-1}+O(1)\quad(z\to0),\quad1\leq i\leq2.
\]
Combined with \eqref{eq:gvarch}, we obtain $e_2m_{1,q}=e_1m_{2,q}$. Since both $\pi_1^*(D)$ and $\pi_2^*(D)$ have supports contained in $W\setminus W^\circ$, we conclude that
\begin{equation}\label{eq:divsame}
e_2\pi_1^*(D)=e_1\pi_2^*(D).
\end{equation}

Let $v\in M_L$ be an arbitrary finite place. By \eqref{eq:Del12}, the function $e_2g_{1,v}-e_1g_{2,v}$ is harmonic on $(W^\circ)_v^{\mathrm{an}}$. Using the boundary expansion of $\psi_v$, we see that $e_2g_{1,v}-e_1g_{2,v}$ extends continuously through every point of $W\setminus W^\circ$. By \cite[\S~1.2.3]{FG22}, the extension of $e_2g_{1,v}-e_1g_{2,v}$ is harmonic. Thus, $e_2g_{1,v}-e_1g_{2,v}$ equals a constant $\kappa_v$ on $W_v^{\mathrm{an}}$ by the maximum principle. For all archimedean $w\in M_L$, set $\kappa_w=0$.

By \eqref{eq:divsame}, the rational section $s=s_1^{\otimes e_2}\otimes s_2^{\otimes(-e_1)}$ has divisor zero and trivializes $H_1^{\otimes e_2}\otimes H_2^{\otimes(-e_1)}$ on $W$. Both $\overline{H}_1$ and $\overline{H}_2$ are adelic by Step 2, and there is a finite subset $S'\subseteq M_L$ containing all places over some place in $S\subseteq M_K$ such that $s$ extends as a nowhere-vanishing model section with norm $1$. Thus, $\kappa_v=0$ for all $v\in M_L\setminus S'$. Define
\[
\kappa:=\sum_{v\in M_L}\frac{[L_v:\Q_v]}{[L:\Q]}\kappa_v\in\R.
\]
Then $e_2h_1-e_1h_2=\kappa$ on $W^\circ(\olQ)$. Evaluating at $w_j$ and letting $j\to\infty$ imply $\kappa=0$. Therefore,
\[
e_2h_1=e_1h_2\quad\text{on }W^\circ(\olQ).
\]
Applying the direction $(2)\Rightarrow(3)$ of \cite[Theorem~B]{FG22} to $(f_1,b_1)$ and $(f_2,b_2)$, by \cite[Theorem~3.39]{FG22} we conclude that
\begin{equation}\label{eq:genintw}
f_{\pi_1(w_\eta)}\approx f_{\pi_2(w_\eta)},
\end{equation}
where $w_\eta$ is the generic point of $W$.

\smallskip

We prove (i) of $(\Delta)$. Let $Z$ be the Zariski closure of $(\varphi\times\varphi)(C\cap B_{\olQ}^2)$, which is an irreducible curve in $V_{\olQ}^2$. Both projections $Z\to V_{\olQ}$ are dominant. By \eqref{eq:genintw}, $Z$ is contained in the Zariski closure $\hat{Z}$ of the intertwined relation on $V_{\olQ}$. Using Proposition~\ref{prop:Md} and Proposition~\ref{prop:rel}(1), it is easy to see that $Z$ is an irreducible component of $\hat{Z}$. By Proposition~\ref{prop:rel}(2), the two pullbacks of $u_V$ on $Z$ coincide, and hence
\[
\pr_1^*(u)=\pr_2^*(u)\quad\text{on }C.
\]

\smallskip

Next we prove (ii) of $(\Delta)$. Since $\sim$ is constructible, by \eqref{eq:genintw} there is a Zariski open dense subset $O\subseteq C\cap B_{\olQ}^2$ contained in $\sim$. The complement $(C\cap B_{\olQ}^2)\setminus O$ is finite, so
\[
J_C=\pr_1\left((C\cap B_{\olQ}^2)\setminus O\right)\subseteq B(\olQ)
\]
is finite and satisfies \eqref{eq:JCprop}.

The proof of $(\Delta)$ is complete.

\medskip

\textbf{Step 4. Inertia and Frobenius.}

Suppose that the conclusion of Theorem~\ref{thm:pair} fails. Then there is a sequence $(t_j)_{j=1}^\infty\subseteq B(\Kab)$ of pairwise distinct points such that $\lim_{j\to\infty}\hat{h}_{f_{t_j}}(a(t_j))=0$. By \eqref{eq:pairheight} and $b=f^{\circ N}(a)$, we have
\[
h_{\overline{H}}(t_j)=(d-1)d^N\hat{h}_{f_{t_j}}(a(t_j))\to0,\quad j\to\infty.
\]

By the Chebotarev density theorem, choose a rational prime $p>\max\{2,N_0\}$ splitting completely in $K$ and a place $v\in M_K\setminus S$ above $p$. Then $K_v=\Q_p$. Fix an extension $\bar{v}$ of $v$ to $\olQ$. After passing to an infinite subsequence, we may assume that either $K(t_j)_{\bar{v}}/\Q_p$ is wildly ramified for all $j\geq1$, or $K(t_j)_{\bar{v}}/\Q_p$ is tamely ramified for all $j\geq1$. Set
\[
P=\Gal(\Kab/K)\cdot\{t_j:j\geq1\}.
\]
Then either $K(t)_{\bar{v}}/\Q_p$ is wildly ramified for all $t\in P$, or $K(t)_{\bar{v}}/\Q_p$ is tamely ramified for all $t\in P$. As $B$ is a geometrically integral curve, by the construction of $P$ we have
\begin{equation}\label{eq:smht}
\overline{\{t\in P:h_{\overline{H}}(t)<\epsilon\}}=\overline{B}_{\olQ},\quad\forall\epsilon>0.
\end{equation}
Set $\delta=p^{-1/(p-1)}\in(0,1)$ and $G=\id_B:B\to B$.

\smallskip

\textbf{Case 1.} Suppose first that $K(t)_{\bar{v}}/\Q_p$ is wildly ramified for all $t\in P$. For every $t\in P$ with $E=K(t)$, by Lemma~\ref{lem:in}(1) choose an order-$p$ inertia element $\tau_E\in\Gal(E/K)$, as in Case 1 of the proof of Theorem~\ref{thm:main}. Define
\[
B_{\mathrm{In}}=\left\lbrace(t,\tau_{K(t)}(t))\colon t\in P\right\rbrace.
\]
As in \S~\ref{sec:infr}, $B_{\mathrm{In}}$ is $\Gal(\olQ/K)$-invariant and $B_{\mathrm{In}}\subseteq T_1(\delta)$. For $G=\id_B$, using \eqref{eq:smht} and the argument in \S~\ref{sec:infr}, we obtain an irreducible bi-finite correspondence $C_{\mathrm{In}}\subseteq Z(B_{\mathrm{In}})$ on $\overline{B}_{\olQ}$, satisfying \eqref{eq:paircorsmall}. Let $J_{C_{\mathrm{In}}}$ be as in $(\Delta)$(ii) of Step 3. Take $(t,\tau_E(t))\in B_{\mathrm{In}}\cap C_{\mathrm{In}}$ with $t\notin J_{C_{\mathrm{In}}}$, where $E=K(t)$. Applying $(\Delta)$(ii), the points
\[
t,\tau_E(t),\ldots,\tau_E^{p-1}(t)
\]
are pairwise distinct and contained in a single $\sim$-class as in \S~\ref{sec:infr}, contradicting $p>N_0$.

\smallskip

\textbf{Case 2.} Suppose that $K(t)_{\bar{v}}/\Q_p$ is tamely ramified for all $t\in P$. As in Case 2 of the proof of Theorem~\ref{thm:main}, choose a Frobenius lift $\sigma$ and define
\[
B_{\mathrm{Fr}}=\left\lbrace(t,\sigma(t))\colon t\in P\right\rbrace.
\]
As in \S~\ref{sec:infr}, $B_{\mathrm{Fr}}$ is $\Gal(\olQ/K)$-invariant and $B_{\mathrm{Fr}}\subseteq T_p(\delta)$. For $G=\id_B$, using \eqref{eq:smht} and the argument in \S~\ref{sec:infr}, we obtain an irreducible bi-finite correspondence $C_{\mathrm{Fr}}\subseteq Z(B_{\mathrm{Fr}})$ on $\overline{B}_{\olQ}$ that satisfies \eqref{eq:paircorsmall}. Let $W\to C_{\mathrm{Fr}}$ be the normalization and $\pi_1,\pi_2:W\to\overline{B}_{\olQ}$ be the two projections. Take a finite extension $L/K$ over which $C_{\mathrm{Fr}}$ is defined. Proposition~\ref{prop:tube} shows $\pi_2^*H\equiv p\pi_1^*H$. By $(\Delta)$(i), the nonconstant function $u\in K(B)\setminus K$ (fixed in Step 3) satisfies $\pi_1^*(u)=\pi_2^*(u)$ in $L(W)$. Then Lemma~\ref{lem:inv} gives a contradiction as in \S~\ref{sec:infr}.
\end{proof}

\subsection*{Acknowledgment}
The author would like to thank Junyi Xie for useful discussions and comments.

\subsection*{AI Disclosure}
The author initially obtained the proof of Theorem~\ref{thm:cyc} using \cite{JXZ2026} (presented in \S~\ref{sec:cyc}). Subsequently, during an exploratory discussion regarding whether the techniques in \cite{JSX26} could be adapted to the abelian case (Theorem~\ref{thm:pcf}), GPT-6 Astra suggested the correspondence method of Ji--Song--Xie \cite[\S~6]{JSX26}. The author then developed the current proof of Theorem~\ref{thm:main}, assisted by GPT-6 Astra with several specific arguments in the proofs of Theorem~\ref{thm:met}, Proposition~\ref{prop:tube}, and Lemma~\ref{lem:base}. In addition, GPT-5.6 Sol was used for literature search, language polishing, and preliminary proofreading. The author independently verified and rewrote all AI-assisted material, and takes full responsibility for the paper's content and correctness.

\end{document}